\documentclass[]{interact}

\usepackage{epstopdf}% To incorporate .eps illustrations using PDFLaTeX, etc.
\usepackage[caption=false]{subfig}% Support for small, `sub' figures and tables
\usepackage{multirow}
\usepackage{enumerate}
\usepackage{color}
\usepackage[numbers,sort&compress]{natbib}% Citation support using natbib.sty
\bibpunct[, ]{[}{]}{,}{n}{,}{,}% Citation support using natbib.sty
\renewcommand\bibfont{\fontsize{10}{12}\selectfont}% Bibliography support using natbib.sty
\makeatletter% @ becomes a letter
\def\NAT@def@citea{\def\@citea{\NAT@separator}}% Suppress spaces between citations using natbib.sty
\makeatother% @ becomes a symbol again

\theoremstyle{plain}% Theorem-like structures provided by amsthm.sty
\newtheorem{theorem}{Theorem}[section]
\newtheorem{lemma}[theorem]{Lemma}

\newtheorem{remark1}{Remark}[section]
\theoremstyle{definition}

\theoremstyle{remark}

\begin{document}

%\jvol{00} \jnum{00} \jyear{2015} \jmonth{February}

\title{Optimal estimators in misspecified linear regression model with an application to real-world data} \break %\%LaTeX\ guide for authors (Style 2 + NLM reference style)}

\author{
\name{Manickavasagar Kayanan\textsuperscript{a,b}$^{\ast}$\thanks{$^\ast$Corresponding author. Email: mgayanan@vau.jfn.ac.lk}
and Pushpakanthie Wijekoon\textsuperscript{c}}
\affil{\textsuperscript{a}Postgraduate Institute of Science, University of Peradeniya, Peradeniya, Sri Lanka; \textsuperscript{b}Deparment of Physical Science, Vavuniya Campus of the University of Jaffna, Vavuniya, Sri Lanka; \textsuperscript{c}Department of Statistics and Computer Science, University of Peradeniya, Peradeniya, Sri Lanka}
%\received{v5.0 released February 2015}
}

\maketitle
\begin{abstract}
In this article, we propose the Sample Information Optimal Estimator (SIOE) and the Stochastic Restricted Optimal Estimator (SROE) for misspecified linear regression model when multicollinearity exists among explanatory variables. Further, we obtain the superiority conditions of proposed estimators over some other existing estimators in the Mean Square Error Matrix (MSEM) criterion in a standard form which can apply to all estimators considered in this study. Finally, a real-world example and a Monte Carlo simulation study are presented for the proposed estimators to illustrate the theoretical results. \end{abstract}


%\begin{classcode}62J05, 62J07\end{classcode}

\section{Introduction}
The multiple linear regression model defined as
\begin{equation}\label{section}
  \bm y=\bm X\bm\beta+\bm\epsilon
\end{equation}
where $\bm y$ is the $n\times 1$ vector of observations on the predictor variable, $\bm X$ is the $n\times p$ matrix of observations on $p$ non stochastic regressor variables, $\bm \beta$ is a $p\times 1$ vectors of unknown parameters, $\bm \epsilon$ is the $n\times 1$ vector of disturbances, such that $E(\bm \epsilon)=0$ and $E(\bm \epsilon\bm \epsilon')=\bm \Omega=\sigma^2 \bm I$.

The estimator for $\bm \beta$ considered commonly in practical situations is the ordinary least squares estimator (OLSE)
\begin{equation}\label{section}
\hat{\bm \beta}_{OLSE}=(\bm X'\bm X)^{-1} \bm X' \bm y
\end{equation}which is unbiased and has the minimum variance among all linear unbiased estimators.

If the columns of the $\bm X$ matrix are nearly linearly dependent, i.e., multicollinear, then the matrix $\bm X'\bm X$  is almost singular. Consequently, the numerical computation of (2) will be unstable, and the variance of the OLSE will be large. As a remedial measure to the multicollinearity problem, biased estimators have been used in the literature. Some of the biased estimators are based only  on model (1), namely Ridge Estimator (RE) \cite{Ho70}, Almost Unbiased Ridge Estimator (AURE) \cite{Si86}, Liu Estimator (LE) \cite{Lu93}, Almost Unbiased Liu Estimator (AULE) \cite{Ak95}, Principal Component Regression Estimator (PCRE) \cite{Ma65},  r-k class estimator \cite{Ba84} and r-d class estimator \cite{Ka01}, and are given as
\begin{align}\label{section}
\hat{\bm \beta}_{RE}=&(\bm X'\bm X+k\bm I)^{-1} \bm X'\bm X\hat{\bm \beta}_{OLSE}\\
\hat{\bm \beta}_{AURE}=&\left(\bm I-k^{2}(\bm X'\bm X+k\bm I)^{-2}\right)\hat{\bm \beta}_{OLSE}\\
\hat{\bm \beta}_{LE}=&(\bm X'\bm X+\bm I)^{-1}(\bm X'\bm X+d\bm I)\hat{\bm \beta}_{OLSE}\\
\hat{\bm \beta}_{AULE}=&\left(\bm I-(1-d)^{2}(\bm X'\bm X+\bm I)^{-2}\right)\hat{\bm \beta}_{OLSE}\\
\hat{\bm \beta}_{PCRE}=& \bm T_{h}\bm T_{h}'\hat{\bm \beta}_{OLSE}\\
\hat{\bm \beta}_{rk}=&\bm T_{h}\bm T_{h}'(\bm X'\bm X+k\bm I)^{-1} \bm X'\bm X\hat{\bm \beta}_{OLSE}\\
\hat{\bm \beta}_{rd}=&\bm T_{h}\bm T_{h}'(\bm X'\bm X+\bm I)^{-1}(\bm X'\bm X+d\bm I)\hat{\bm \beta}_{OLSE}					
\end{align}respectively, where $k>0$ and  $0<d<1$ are the shrinkage parameters, and $\bm T_{h}=(t_{1},t_{2},...,t_{h})$ is the first $h$ columns of the standardized eigenvectors of $\bm X'\bm X$ represents by $T=(t_{1},t_{2},...,t_{h},...,t_{m})$.

According to the literature, some other biased estimators are also available based on model (1) and prior information about $\bm \beta$ in the form of exact linear restrictions or stochastic linear restrictions. Theil and Goldberger \cite{th61} have presented the stochastic linear restrictions on $\bm \beta$ as
\begin{equation}\label{section}
\bm r=\bm R \bm\beta+\bm v					
\end{equation}where $\bm r$ is the $q\times 1$ vector, $\bm R$ is the given $q\times l$ matrix with rank $q$, and $\bm v$ is the $q\times 1$ random vector of disturbances, such that $E(\bm v)=0$, $E(\bm v\bm v')=\bm \Psi=\sigma^2 \bm W$ with $\bm W$ is positive definite and $E(\bm v \bm u')=0$. Here, equation (10) represents exact linear restrictions when the random error term $\bm v$ vanishes.

Thiel and Goldberger \cite{th61} then proposed the Mixed Regression Estimator (MRE) by combining model (1) and stochastic linear restrictions (10) as
\begin{equation}\label{section}
\hat{\bm \beta}_{MRE}=(\bm X'\bm \Omega^{-1} \bm X+\bm R' \bm\Psi^{-1} \bm R)^{-1} (\bm X'\bm\Omega^{-1} \bm y+\bm R' \bm\Psi^{-1} \bm r)
\end{equation}
Further, to improve the performance of MRE, several researchers have introduced different types of stochastic restricted estimators such as Stochastic Restricted Liu Estimator (SRLE) \cite{Hu04}, Stochastic Restricted Ridge Estimator (SRRE) \cite{Li10}, Stochastic Restricted Almost Unbiased Ridge Estimator (SRAURE) \cite{Ji14}, Stochastic Restricted Almost Unbiased Liu Estimator (SRAULE) \cite{Ji14}, Stochastic Restricted Principal Component Regression Estimator (SRPCRE) \cite{He14}, Stochastic Restricted r-k class estimator (SRrk) \cite{Jib14} and Stochastic Restricted r-d class estimator (SRrd) \cite{Jib14}, and are given as

\begin{align}\label{section}
\hat{\bm \beta}_{SRRE}=&(\bm X'\bm X+k\bm I)^{-1} \bm X'\bm X\hat{\bm \beta}_{MRE}\\
\hat{\bm \beta}_{SRAURE}=&\left(\bm I-k^{2}(\bm X'\bm X+k\bm I)^{-2}\right)\hat{\bm \beta}_{MRE}\\
\hat{\bm \beta}_{SRLE}=&(\bm X'\bm X+\bm I)^{-1}(\bm X'\bm X+d\bm I)\hat{\bm \beta}_{MRE}\\
\hat{\bm \beta}_{SRAULE}=&\left(\bm I-(1-d)^{2}(\bm X'\bm X+\bm I)^{-2}\right)\hat{\bm \beta}_{MRE}\\
\hat{\bm \beta}_{SRPCRE}=& \bm T_{h}\bm T_{h}'\hat{\bm \beta}_{MRE}\\
\hat{\bm \beta}_{SRrk}=&\bm T_{h}\bm T_{h}'(\bm X'\bm X+k\bm I)^{-1} \bm X'\bm X\hat{\bm \beta}_{MRE}\\
\hat{\bm \beta}_{SRrd}=&\bm T_{h}\bm T_{h}'(\bm X'\bm X+\bm I)^{-1}(\bm X'\bm X+d\bm I)\hat{\bm \beta}_{MRE}					
\end{align}respectively.

The superiority of the biased estimators are usually determined based on the Mean Square Error Matrix (MSEM) criterion, and it is defined as
\begin{equation}\label{section}
\begin{split}
   MSEM(\hat{\bm \beta})=&E(\hat{\bm \beta}-\bm \beta)(\hat{\bm \beta}-\bm \beta)'\\
   =&D(\hat{\bm \beta})+Bias(\hat{\bm \beta})Bias(\hat{\bm \beta})'
\end{split}
\end{equation}
where $\hat{\bm \beta}$ is the biased estimator, $D(\hat{\bm \beta})$ is the dispersion matrix of $\hat{\bm \beta}$, and $Bias(\hat{\bm \beta})=E(\hat{\bm \beta})-\bm \beta$ is the bias vector of $\hat{\bm \beta}$.

The exclusion of some relevant explanatory variables in linear regression model is addressed as another critical issue in the econometric research. Suppose the correct model (1) is written as
\begin{equation}\label{section}
 \bm y=\bm X_{1}\bm \beta_{1}+\bm X_{2}\bm \beta_{2}+\bm \epsilon
\end{equation}where $\bm X_{1}$ and $\bm X_{2}$ are the $n\times l$ and $n\times p$ matrices of observations on the $m=l+p$ regressors, $\bm \beta_{1}$ and $\bm \beta_2$ are the $l\times 1$ and $p\times 1$ vectors of unknown coefficients. If the researcher misspecifies the regression model (20) by excluding $p$ regressors as
\begin{equation}\label{section}
 \bm y=\bm X_{1}\bm \beta_{1}+\bm u
\end{equation}where $\bm u=\bm X_{2}\bm \beta_{2}+\bm \epsilon$, then model (21) is a misspecified regression model. Since $E(\bm u)\neq0$ and $\bm X_2$ may have some correlation with $\bm X_1$ if the columns of the $\bm X$ matrix are multicollinear, then one or more assumptions of the linear regression model will be violated, which leads to the biased and inconsistent estimation of parameters.

Several researchers have examined the consequences of the estimation procedure in this situation. Sarkar \cite{Sa89} compared the performance of OLSE, RE and PCRE when multicollinearity exists in a misspecified regression model. $\c{S}$iray \cite{Sir15}, Wu \cite{Wu16} and Chandra and Tyagi \cite{Ch17} examined the efficiency of the r-d class estimator and r-k class estimator over some existing estimators in the misspecified regression model. Recently, Kayanan and Wijekoon \cite{Kay17a} studied the performance RE, AURE, LE, AULE, PCRE, r-k class estimator and r-d class estimator in the misspecified linear regression model. These studies have shown that r-d class estimator and r-k class estimator outperformed the other biased estimators in the misspecified regression model for the selected range of shrinkage parameters when multicollinearity exists among the regressor variables.

Several researchers have studied the performance of stochastic restricted estimators when model (20) is misspecified by excluding $p$ regressors. Ter\"{a}svirta \cite{Te80}, and Hubert and Wijekoon \cite{Hu04} have written the stochastic linear restrictions (10) for the misspecified regression model (21) as
\begin{equation}\label{section}
\bm r=\bm R \bm\beta_1+\bm g+\bm v					
\end{equation}where $\bm g$ is the $q\times 1$ unknown fixed vector.

Ter\"{a}svirta \cite{Te80}, Mittelhmmer \cite{Mt81}, Ohtani and Honda \cite{Oh84}, Kadiyala \cite{Ka86}, Trenkler and Wijekoon \cite{Tr89} and Wijekoon and Trenkler \cite{Wt89} have compared the MRE with the OLSE under misspecified regression model when $p$ regressors are excluded from the model. Further, Hubert and Wijekoon \cite{Hu04} discussed the performance of SRLE over the MRE in the misspecified model. Kayanan and Wijekoon \cite{Kay17b} examined the performance of SRRE, SRAURE, SRAULE, SRPCRE, SRrk and SRrd over the other existing estimators in the misspecified regression model. These studies have shown that SRLE and SRRE outperformed the other stochastic restricted estimators in the misspecified regression model for the selected range of shrinkage parameters when multicollinearity exists among the regressor variables.

Arumairajan and Wijekoon \cite{Ar15} have proposed Generalized Optimal Estimator (GOE) based on MSEM of the RE, AURE, LE and AULE for the correctly specified regression model, and they have shown that GOE outperformed those estimators. Their work motivated us to study the optimal estimators under misspecified regression model by considering MSEM of the RE, AURE, LE, AULE, PCRE, r-k class estimator, r-d class estimator, SRRE, SRAURE, SRLE, SRAULE, SRPCRE, SRrk and SRrd.

The main contributions of this work are to introduce two new estimators; the Sample Information Optimal Estimator (SIOE) and the Stochastic Restricted Optimal Estimator (SROE); for the misspecified linear regression model, and to obtain a common form of superiority conditions of the proposed estimators over some existing estimators based on the MSEM criterion. Further, we employ a Monte Carlo simulation study and a real-world example to illustrate the theoretical results. The rest of the article is organized as follows. The canonical form of the misspecified model and estimators are given in section 2. The proposed optimal estimators and MSEM comparisons are presented in section 3. In section 4, a Monte Carlo simulation study and a real-world example to illustrate the theoretical results are given. Section 5 includes concluding remarks. References and Appendix are given at the end of this article.

\section{Canonical form and estimators}
 Suppose the correct regression model is given in (20), and further, it is misspecified as (21). \subsection{Biased estimators}
To get simplified expressions, we apply spectral decomposition to the symmetric matrix $\bm X_{1}'\bm X_{1}$ \cite{Si86}, since $\bm X_{1}'\bm X_{1}$ is a positive definite matrix. Then we have $\bm T'\bm X_{1}'\bm X_{1}\bm T=\bm \Lambda=diag(\lambda_{1},......,\lambda_{l} )$, where $\bm T=(t_{1},t_{2},......,t_{l})$ is the orthogonal matrix and $\lambda_{i}>0$ being the $i^{th}$ eigenvalue of $\bm X_{1}'\bm X_{1}$. Let $\bm T_h=(t_{1},t_{2},......,t_{h})$ be the remaining column of $\bm T$ having deleted $l-h$ columns where  $h\leq l$. Hence, $\bm T_{h}\bm X_{1}'\bm X_{1}\bm T_{h}=\bm \Lambda_{h}=diag(\lambda_{1},......,\lambda_{h}  )$.

Let $\bm Z=\bm X_{1}\bm T$ and $\bm \gamma=\bm T'\bm \beta_{1}$ then models (20) and (21) can be written in canonical form as

\begin{align}\label{section}
  \bm y=&\bm Z\bm \gamma+\bm \delta+\bm \epsilon\\
  \bm y=&\bm Z\bm \gamma+\bm u
\end{align}respectively, where $\bm \delta=\bm X_{2}\bm \beta_{2}$. Note that when $\bm \delta=0$ the model is correctly specified.

The OLSE of model (24) is given by
\begin{equation}\label{section}
\hat{\bm \gamma}_{OLSE}=(\bm Z'\bm Z)^{-1} \bm Z' \bm y=\bm \Lambda^{-1} \bm Z'\bm y					
\end{equation}

According to Kayanan and Wijekoon \cite{Kay17a}, the generalized form to represent the estimators RE, AURE, LE, AULE, PCR, $r-k$ class estimator and $r-d$ class estimator for model (24) is given by
\begin{equation}\label{section}
\hat{\bm \gamma}_G=\bm G \hat{\bm \gamma}_{OLSE}	
\end{equation}where
\[\hat{\bm \gamma}_G=\begin{cases}
   \hat{\bm \gamma}_{RE}      & \quad \text{if } \bm G=(\bm \Lambda+k\bm I)^{-1} \bm \Lambda  \\
    \hat{\bm \gamma}_{AURE}&\quad \text{if } \bm G=\left(\bm I-k^{2} (\bm \Lambda+k\bm I)^{-2} \right)  \\
    \hat{\bm \gamma}_{LE}&\quad \text{if } \bm G=(\bm \Lambda+\bm I)^{-1} (\bm \Lambda+d\bm I)\\
    \hat{\bm \gamma}_{AULE}&\quad \text{if } \bm G=\left(\bm I-(1-d)^{2} (\bm \Lambda+\bm I)^{-2} \right)  \\
    \hat{\bm \gamma}_{PCRE}&\quad \text{if } \bm G=\bm T_{h} \bm T_{h}'\\
    \hat{\bm \gamma}_{rk}&\quad \text{if } \bm G=\bm T_{h}\bm  T_{h}' (\bm \Lambda+k\bm I)^{-1} \bm \Lambda\\
    \hat{\bm \gamma}_{rd}&\quad \text{if } \bm G=\bm T_{h} \bm T_{h}'(\bm \Lambda+\bm I)^{-1} (\bm \Lambda+d\bm I)
  \end{cases}\]having the shrinkage parameters $k>0$ and  $0<d<1$.

\subsection{Stochastic Restricted estimators}
To get simplified expressions, we apply the simultaneous decomposition for the symmetric matrices $\bm X_{1}'\bm X_{1}$ and $\bm R'\bm \Psi^{-1}\bm R$ \cite{Ka86}, where $\bm X_{1}'\bm X_{1}$ is a positive definite matrix and $\bm R'\bm \Psi^{-1}\bm R$ is a positive semi-definite matrix. Then we have
\[\bm B'\bm X_{1}'\bm X_{1}\bm B=\bm I \qquad  \text{and} \qquad   \bm B'\bm R'\bm \Psi^{-1}\bm R\bm B=\bm \Lambda_{*}\]
where $\bm B$ is a $l\times l$ nonsingular matrix,  $\bm \Lambda_{*}$ is a $l\times l$ diagonal matrix with eigenvalues $\lambda^{*}_{i}>0$ for $i=1,2,......,q$ and $\lambda^{*}_{i}=0$ for $i=q+1,......,l$.

Let $\bm Z_{*}=\bm X_{1}\bm B$,   $\bm R_{*}=\bm R\bm B$,  $\bm \gamma_{*}=\bm B^{-1}\bm \beta_{1}$,   $\bm Z_{*}'\bm Z_{*}=\bm I$   and   $\bm R_{*}'\bm \Psi^{-1}\bm R_{*}=\bm \Lambda_{*}$ then models (20), (21) and (22) can be written as
\begin{align}\label{section}
\bm y=&\bm Z_{*} \bm \gamma_{*}+\bm \delta+\bm \epsilon\\				
\bm y=&\bm Z_{*} \bm \gamma_{*}+\bm u				\\			
\bm r=&\bm R_{*}\bm \gamma_{*}+\bm g+\bm v						
\end{align}respectively. Now the MRE of the model (28) can be written as
\begin{equation}\label{section}
\begin{split}
\bm \gamma_{MRE}=&(\bm Z_{*}' \bm Z_{*}+\bm R_{*}' \bm \Psi^{-1} \bm R_{*} )^{-1} (\bm Z_{*}' \bm y+\bm R_{*}'\bm W^{-1}\bm r)\\
=&(\bm I+\sigma^{2} \bm\Lambda_{*} )^{-1} (\bm Z_{*}' \bm y+\bm R_{*}'\bm W^{-1} \bm r)
\end{split}			
\end{equation}

By following Kayanan and Wijekoon \cite{Kay17b}, the generalized form to represent the stochastic restricted estimators SRRE, SRAURE, SRLE, SRAULE, SRPCR, SRrk and SRrd for  model (28) is given by
\begin{equation}\label{section}
\hat{\bm \gamma}^{*}_G=\bm G_{*} \hat{\bm \gamma}_{MRE}	
\end{equation}where
\[\hat{\bm \gamma}^{*}_G=\begin{cases}
   \hat{\bm \gamma}_{SRRE}      & \quad \text{if } \bm G_{*}=(1+k)^{-1}    \\
    \hat{\bm \gamma}_{SRAURE}&\quad \text{if } \bm G_{*}=(1+k)^{-2} (1+2k)    \\
    \hat{\bm \gamma}_{SRLE}&\quad \text{if } \bm G_{*}=2^{-1} (1+d)  \\
    \hat{\bm \gamma}_{SRAULE}&\quad \text{if } \bm G_{*}= 2^{-2} (1+d)(3-d)\\
    \hat{\bm \gamma}_{SRPCRE}&\quad \text{if } \bm G_{*}=\bm T_{h} \bm T_{h}'\\
    \hat{\bm \gamma}_{SRrk}&\quad \text{if } \bm G_{*}=(1+k)^{-1}\bm T_{h} \bm T_{h}' \\
    \hat{\bm \gamma}_{SRrd}&\quad \text{if } \bm G_{*}=2^{-1} (1+d)\bm T_{h}\bm T_{h}'
  \end{cases}\] having the shrinkage parameters $k>0$ and  $0<d<1$.

\subsection{Stochastic properties of the estimators}
Now we define a common from
\begin{equation}\label{section}
\hat{\bm \gamma}_{\tilde{ G}}=\tilde{\bm G}\tilde{\bm \gamma}
\end{equation}to represent both $\hat{\bm\gamma}_{G}$ and $\hat{\bm\gamma}^{*}_{G}$. Note that, $\hat{\bm \gamma}_{\tilde{ G}}=\hat{\bm\gamma}_{G}$ when $\tilde{\bm G}=\bm G$ and $\tilde{\bm\gamma}=\bm\gamma$, and $\hat{\bm\gamma}_{\tilde{G}}=\hat{\bm\gamma}^{*}_{G}$ when $\tilde{\bm G}=\bm G_{*}$ and $\tilde{\bm\gamma}=\bm\gamma_{*}$.

According to Kayanan and Wijekoon \cite{Kay17a,Kay17b}, the bias vector, dispersion matrix and MSEM of $\hat{\bm \gamma}_{\tilde{ G}}$ can be presented as
\begin{align}\label{section}
Bias(\hat{\bm \gamma}_{\tilde{ G}} )=&\tilde{\bm G}(\tilde{\bm \gamma}+\bm \tau \bm A) -\tilde{\bm \gamma}			 	\\			
D(\hat{\bm \gamma}_{\tilde{G}} ) =&\sigma^{2} \tilde{\bm G}\bm \tau \tilde{\bm G}'\\
MSEM(\hat{\bm \gamma}_{\tilde{G}} )=&\sigma^{2} \tilde{\bm G} \bm \tau \tilde{\bm G}'+\left(\tilde{\bm G}(\tilde{\bm \gamma}+\bm \tau \bm A) -\tilde{\bm \gamma}\right)\left(\tilde{\bm G}(\tilde{\bm \gamma}+\bm \tau\bm A) -\tilde{\bm \gamma}\right)'\end{align}respectively. Note that, equations (33), (34) and (35) give the bias vector, dispersion matrix and MSEM for the
\begin{enumerate}[(i)]
  \item RE, AURE, LE, AULE, PCR, r-k class estimator and r-d class estimator when $\tilde{\bm G}=\bm G$, $\tilde{\bm\gamma}=\bm\gamma$, $\bm\tau=\bm\Lambda^{-1}$ and  $\bm A=\bm Z' \bm\delta$, and
\item SRRE, SRAURE, SRLE, SRAULE, SRPCR, SRrk and SRrd when $\tilde{\bm G}=\bm G_{*}$, $\tilde{\bm\gamma}=\bm\gamma_{*}$, $\bm\tau=(\bm I+\sigma^2 \bm\Lambda_{*} )^{-1}$ and  $\bm A=(\bm Z_{*}' \bm \delta+\bm R_{*}' \bm W^{-1} \bm g)$.\end{enumerate}
 The respective expressions for the each estimator are summarized in Tables C1 and C2 in Appendix C.
\section{Optimal estimators}
Note that the Scalar Mean Square Error (SMSE) of $\hat{\bm \gamma}_{\tilde{G}}$ in the common form can be written as
\begin{equation}\label{section}
\begin{split}
   SMSE(\hat{\bm \gamma}_{\tilde{G}})=&tr\left(MSEM(\hat{\bm\gamma}_{\tilde{G}} )\right)\\
   =&\sigma^{2} tr\left(\tilde{\bm G}\bm \tau \tilde{\bm G}'\right)+\left(\tilde{\bm G}(\tilde{\bm\gamma}+\bm\tau \bm A) -\tilde{\bm\gamma}\right)'\left(\tilde{\bm G}(\tilde{\bm \gamma}+\bm\tau \bm A) -\tilde{\bm\gamma}\right)\\
   =&\sigma^{2} tr\left(\tilde{\bm G} \bm\tau \tilde{\bm G}'\right)+(\tilde{\bm\gamma}+\bm\tau \bm A)'\tilde{\bm G}'\tilde{\bm G}(\tilde{\bm\gamma}+\bm\tau \bm A)-(\tilde{\bm\gamma}+\bm\tau \bm A)'\tilde{\bm G}'\tilde{\bm \gamma}\qquad\\
   &-\tilde{\bm\gamma}'\tilde{\bm G}(\tilde{\bm\gamma}+\bm\tau \bm A)+\tilde{\bm\gamma}'\tilde{\bm\gamma}
\end{split}
\end{equation}
Now by differentiating (36) with respect to $\tilde{\bm G}$ we obtain (refer Appendix A for matrix operations)
\begin{equation}\label{section}
\begin{split}
\frac{\partial\left(SMSE(\hat{\bm\gamma}_G)\right)}{\partial \tilde{\bm G}}=&\sigma^{2} \tilde{\bm G} (\bm\tau+\bm\tau' )+2\tilde{\bm G} (\tilde{\bm\gamma}+\bm\tau \bm A) (\tilde{\bm\gamma}+\bm\tau \bm A)'-(\tilde{\bm\gamma}+\bm\tau \bm A)\tilde{\bm\gamma}'-\tilde{\bm\gamma}(\tilde{\bm\gamma}+\bm\tau \bm A)'\qquad\\
=&2\tilde{\bm G}\left(\sigma^{2}\bm\tau+(\tilde{\bm\gamma}+\bm\tau \bm A) (\tilde{\bm\gamma}+\bm\tau \bm A)'\right) -(\tilde{\bm\gamma}+\bm\tau \bm A)\tilde{\bm\gamma}'-\tilde{\bm\gamma}(\tilde{\bm\gamma}+\bm\tau \bm A)'
\end{split}
\end{equation}
Since $\bm \tau$ is symmetric and  positive definite matrix, then $\sigma^{2}\bm\tau+(\tilde{\bm\gamma}+\bm\tau \bm A) (\tilde{\bm\gamma}+\bm\tau \bm A)'$ is positive definite \cite[see][p.366]{ro95}.

Equating (37) to null matrix, we can find the optimum $\tilde{\bm G}$, which is
\begin{equation}\label{section}
\tilde{\bm G}_{opt}=2^{-1}\left((\tilde{\bm \gamma}+\bm\tau \bm A)\tilde{\bm\gamma}'+\tilde{\bm\gamma}(\tilde{\bm\gamma}+\bm\tau \bm A)'\right)\left(\sigma^{2}\bm\tau+(\tilde{\bm\gamma}+\bm\tau \bm A) (\tilde{\bm\gamma}+\bm\tau \bm A)'\right)^{-1}
\end{equation}
Note that, the only unknown parameter in the above equation is $\tilde{\bm \gamma}$.\\

By substituting $\tilde{\bm\gamma}=\bm\gamma$, $\bm\tau=\bm\Lambda^{-1}$ and  $\bm A=\bm Z'\bm\delta$ in equation (38), now we define the Sample Information Optimal estimator (SIOE) as
\begin{equation}\label{section}
\hat{\bm\gamma}_{SIOE}=\bm G_{opt}\hat{\bm\gamma}_{OLSE}					
\end{equation} where \\$\bm G_{opt}=2^{-1}\left((\bm\gamma+\bm\Lambda^{-1} \bm Z'\bm\delta)\bm\gamma'+\bm\gamma(\bm\gamma+\bm\Lambda^{-1} \bm Z'\bm\delta)'\right)\left(\sigma^{2}\bm\Lambda^{-1}+(\bm\gamma+\bm\Lambda^{-1} \bm Z'\bm\delta) (\bm\gamma+\bm\Lambda^{-1} \bm Z'\bm\delta)'\right)^{-1}$.\\

Further, by substituting $\tilde{\bm\gamma}=\bm\gamma_{*}$, $\bm\tau=(\bm I+\sigma^2 \bm\Lambda_{*} )^{-1}$ and  $\bm A=(\bm Z_{*}' \bm\delta+\bm R_{*}' \bm W^{-1} \bm g)$ in equation (38), we define the Stochastic Restricted Optimal estimator (SROE) as
\begin{equation}\label{section}
\hat{\bm\gamma}_{SROE}=\bm G^{*}_{opt}\hat{\bm\gamma}_{MRE}					
\end{equation}  where\\ $\bm G^{*}_{opt}=2^{-1}\{(\bm\gamma_{*}+(\bm I+\sigma^2 \bm\Lambda_{*} )^{-1} (\bm Z_{*}' \bm\delta+\bm R_{*}' \bm W^{-1} \bm g))\bm\gamma_{*}'+\bm\gamma_{*}(\bm\gamma_{*}+(\bm I+\sigma^2 \bm\Lambda_{*} )^{-1}(\bm Z_{*}' \bm\delta+\bm R_{*}' \bm W^{-1} \bm g))'\}(\sigma^{2}(\bm I+\sigma^2 \bm\Lambda_{*} )^{-1}+(\bm\gamma_{*}+(\bm I+\sigma^2 \bm\Lambda_{*} )^{-1} (\bm Z_{*}' \bm\delta+\bm R_{*}' \bm W^{-1} \bm g)) (\bm \gamma_{*}+(\bm I+\sigma^2 \bm\Lambda_{*} )^{-1} (\bm Z_{*}' \bm\delta+\bm R_{*}' \bm W^{-1} \bm g))')^{-1}$.

In equation (38), $\tilde{\bm\gamma}$ may be either $\bm\gamma=\bm T'\bm\beta_{1}$ or $\bm\gamma_{*}=\bm B^{-1}\bm\beta_{1}$. Since $\bm\beta_{1}$ is an unknown parameter in model (21), it is necessary to identify an estimated value for $\bm \beta_{1}$ to substitute $\bm G_{opt}$ and $ \bm G^{*}_{opt}$ in equation (39) and (40), respectively, when estimating SIOE and SROE. According to the method suggested by Newhouse and Oman \cite{new71}, if the MSEM is a function of the true regression coefficient vector $\bm \beta$, the error variance $\sigma^2$ and shrinkage parameter $k$, then the MSEM can be minimized when $\bm\beta$ is the normalized eigenvector corresponding to the largest eigenvalue of $\bm X'\bm X$ matrix which satisfy the constraint $\bm\beta'\bm\beta=1$, where $\bm X$ is the standardized matrix of regressor variables. Following this approach, first, we have to standardize the regressor variables before estimating SIOE and SROE to select the vector  $\bm\beta_{1}$.

Now, the bias vector, dispersion matrix and MSEM of SIOE and SROE can be obtained by substituting
\begin{enumerate}[(i)]
  \item $\hat{\bm \gamma}_{\tilde{G}}=\hat{\bm\gamma}_{SIOE}$, $\tilde{\bm G}=\bm G_{opt}$, $\tilde{\bm\gamma}=\bm\gamma$, $\bm\tau=\bm\Lambda^{-1}$ and  $\bm A=\bm Z'\bm\delta$, and
  \item $\hat{\bm \gamma}_{\tilde{G}}=\hat{\bm\gamma}_{SROE}$, $\tilde{\bm G}=\bm G^{*}_{opt}$, $\tilde{\bm\gamma}=\bm\gamma_{*}$, $\bm\tau=(\bm I+\sigma^2 \bm\Lambda_{*} )^{-1}$ and  $\bm A=(\bm Z_{*}' \bm\delta+\bm R_{*}' \bm W^{-1} \bm g)$,
\end{enumerate} respectively, to equations (33), (34) and (35).
\begin{remark1}
  Note that when $\bm\delta=0$ in equation (39), $\bm G_{opt}=\tilde{\bm\gamma}\tilde{\bm\gamma}'\left(\sigma^{2}\bm\Lambda^{-1}+\tilde{\bm\gamma} \tilde{\bm\gamma}'\right)^{-1}$. This is the SIOE for the correctly specified model introduced by Arumairajan and Wijekoon \cite{Ar15}.
\end{remark1}
\begin{remark1}
  When $\bm\delta=0$ and $\bm g=0$ in equation (40), we can obtain the SROE for the correctly specified model.
\end{remark1}
\subsection{Mean Square Error Matrix (MSEM) comparison}
Now we state the following theorems to present the superiority conditions of $\hat{\bm\gamma}_{SIOE}$ and $\hat{\bm\gamma}_{SROE}$ over $\hat{\bm\gamma}_G$ and $\hat{\bm\gamma}^{*}_G$, respectively, in the MSEM criterion.
\begin{theorem}
  If the largest eigenvalue of the matrix $\bm G_{opt} \bm\tau \bm G_{opt}' (\bm G\bm\tau \bm G' )^{-1}$ is less than one, $\hat{\bm\gamma}_{SIOE}$ is superior to $\hat{\bm\gamma}_G$ if and only if $(\bm G_{opt} (\bm\gamma+\bm\tau \bm A)-\bm\gamma)'\left(\sigma^{2} (\bm G \bm\tau \bm G'-\bm G_{opt} \bm\tau \bm G_{opt}' )+(\bm G (\bm\gamma+\bm\tau \bm A)-\bm\gamma) (\bm G (\bm\gamma+\bm \tau \bm A)-\bm\gamma)' \right)^{-1} (\bm G_{opt} (\bm\gamma+\bm\tau \bm A)-\bm\gamma)\leq1$, where $\bm\tau=\bm\Lambda^{-1}$ and  $\bm A=\bm Z'\bm \delta$.
\end{theorem}
\begin{proof}
  Based on equation (34) now we obtain
  \begin{equation}%\label{section}
    D(\hat{\bm \gamma}_G ) -D(\hat{\bm\gamma}_{SIOE} ) =\sigma^{2} \left(\bm G \bm\tau \bm G'-\bm G_{opt} \bm\tau \bm G_{opt}' \right)
  \end{equation}

Note that $\bm G\bm\tau \bm G'>0$ and  $\bm G_{opt}\bm\tau \bm G_{opt}'>0$ \cite[see][p.366]{ro95}. According to Lemma B.2 in Appendix B, $G\tau G'-G_{opt} \tau G_{opt}' >0$ if the largest eigenvalue of the matrix $\bm G_{opt} \bm\tau \bm G_{opt}' (\bm G\bm\tau \bm G' )^{-1}$ is less than one.
Then according to Lemma B.1 in Appendix B, $MSEM(\hat{\bm\gamma}_G )-MSEM(\hat{\bm\gamma}_{SIOE} )$ is nonnegative definite if $(\bm G_{opt} (\bm\gamma+\bm\tau \bm A)-\bm\gamma)'\left(\sigma^{2} (\bm G \bm\tau \bm G'-\bm G_{opt} \bm\tau \bm G_{opt}' )+(\bm G (\bm\gamma+\bm\tau \bm A)-\bm\gamma) (\bm G (\bm\gamma+\bm \tau \bm A)-\bm\gamma)' \right)^{-1} (\bm G_{opt} (\bm\gamma+\bm\tau \bm A)-\bm\gamma)\leq1$, where $\bm\tau=\bm\Lambda^{-1}$ and  $\bm A=\bm Z'\bm \delta$. This completes the proof.
\end{proof}

\begin{theorem}
  If the largest eigenvalue of the matrix $\bm G^{*}_{opt} \bm\tau \bm G^{*'}_{opt} \bm (G_{*} \bm \tau \bm G_{*}' )^{-1}$ is less than one, $\hat{\bm\gamma}_{SROE}$ is superior to $\hat{\bm\gamma}_G$ if and only if $(\bm G^{*}_{opt} (\bm\gamma_{*}+\bm\tau \bm A)-\bm\gamma_{*})'\left(\sigma^{2} (\bm G_{*} \bm\tau \bm G_{*}'-\bm G^{*}_{opt} \bm\tau \bm G^{*'}_{opt} )+(\bm G_{*} (\gamma_{*}+\bm\tau \bm A)-\bm\gamma_{*}) (\bm G_{*} (\bm\gamma_{*}+\bm\tau \bm A)-\bm\gamma_{*})' \right)^{-1} (\bm G^{*}_{opt} (\bm\gamma_{*}+\bm\tau \bm A)-\bm\gamma_{*})\leq1$, where $\bm \tau=(\bm I+\sigma^2 \bm\Lambda_{*} )^{-1}$ and  $\bm A=(\bm Z_{*}' \bm\delta+\bm R_{*}' \bm W^{-1} \bm g)$.
\end{theorem}
\begin{proof}
  Again Based on equation (34) we obtain
  \begin{equation}%\label{section}
    D(\hat{\bm\gamma}_G ) -D(\hat{\bm\gamma}_{SROE} ) =\sigma^{2} \left(\bm G_{*} \bm\tau \bm G_{*}'-\bm G^{*}_{opt} \bm\tau \bm G^{*'}_{opt} \right)
  \end{equation}

Since $\bm G_{*}\bm\tau \bm G_{*}'>0$ and  $\bm G^{*}_{opt}\bm\tau \bm G^{*'}_{opt}>0$ \cite[see][p.366]{ro95}, $\bm G_{*} \bm\tau \bm G_{*}'-\bm G^{*}_{opt} \bm\tau \bm G^{*'}_{opt} >0$ if the largest eigenvalue of the matrix $\bm G^{*}_{opt} \bm\tau \bm G^{*'}_{opt} (\bm G_{*} \bm\tau \bm G_{*}' )^{-1}$ is less than one (see Lemma B.2 in Appendix B).
Then according to Lemma B.1 in Appendix B, $MSEM(\hat{\bm \gamma}_G )-MSEM(\hat{\bm \gamma}_{SROE} )$ is nonnegative definite if $(\bm G^{*}_{opt} (\bm\gamma_{*}+\bm\tau \bm A)-\bm\gamma_{*})'\left(\sigma^{2} (\bm G_{*} \bm\tau \bm G_{*}'-\bm G^{*}_{opt} \bm\tau \bm G^{*'}_{opt} )+(\bm G_{*} (\gamma_{*}+\bm\tau \bm A)-\bm\gamma_{*}) (\bm G_{*} (\bm\gamma_{*}+\bm\tau \bm A)-\bm\gamma_{*})' \right)^{-1} (\bm G^{*}_{opt} (\bm\gamma_{*}+\bm\tau \bm A)-\bm\gamma_{*})\leq1$, where $\bm \tau=(\bm I+\sigma^2 \bm\Lambda_{*} )^{-1}$ and  $\bm A=(\bm Z_{*}' \bm\delta+\bm R_{*}' \bm W^{-1} \bm g)$. This completes the proof.
\end{proof}
Note that the superiority conditions of SIOE over the biased estimators RE, AURE, LE, AULE, PCRE, r-k class estimator and r-d class estimator can obtain by substituting appropriate expressions in Theorem 3.1. Similarly, the superiority conditions of SROE over the biased estimators SRRE, SRAURE, SRLE, SRAULE, SRPCRE, SRrk and SRrd can obtain by substituting appropriate expressions in Theorem 3.2.
\section{Illustrations of theoretical results }
\subsection{Monte Carlo simulation study}
According to McDonald and Galarneau \cite{Mc75}, now we generate the regressor variables as follows:
\begin{equation}\label{section}
x_{i,j}=\sqrt{(1-\alpha^{2})} z_{i,j}+\alpha z_{i,6}  \qquad;i=1,2,......,n.\;  j=1,2,3,4,5.
\end{equation}
where $z_{i,j}$ is an independent standard normal pseudo random number, and $\alpha$ is specified so that the theoretical correlation between any two explanatory variables is given by $\alpha^2$. A predictor variable is generated by using the following equation
\begin{equation}\label{section}
y_{i}=\beta_{1} x_{i,1}+\beta_{2} x_{i,2}+\beta_{3} x_{i,3}+\beta_{4} x_{i,4}+\beta_{5} x_{i,5}+\epsilon_{i}  \qquad;i=1,2,......,n.
\end{equation}
where $\epsilon_{i}$ is a normal pseudo random number with mean zero and variance one. Also, we choose $\bm \beta=(\beta_{1}, \beta_{2} , \beta_{3}, \beta_{4},\beta_{5})$ as the normalized eigenvector corresponding to the largest eigenvalue of $\bm X'\bm X$ for which $\bm \beta'\bm\beta=1$. Further, we choose $\bm R=(1,1,1,1,1)$  and $\bm g=(0,0,0,0,0)$.
To investigate the effects of different degrees of multicollinearity on the estimators, we choose $\alpha=(0.9,0.99,0.999)$, and to study effect of misspecification, we choose $\bm X_{1}=(x_{1},x_{2},x_{3})$ and $\bm X_{2}=(x_{4},x_{5})$.
For simplicity, we select values $k$ and $d$ in the range(0,1).

The simulation is repeated 2000 times by generating new pseudo random numbers and the simulated SMSE values of the estimators are obtained using the following equation:
\begin{equation}\label{section}
SMSE(\hat{\bm \gamma} )=\frac{1}{2000} \sum_{j=1}^{2000}tr\left(MSEM(\hat{\bm\gamma}_{j})\right)
\end{equation} where $\hat{\bm \gamma}$ represents $\hat{\bm\gamma}_{G}$, $\hat{\bm\gamma}^{*}_{G}$, $\hat{\bm\gamma}_{SIOE}$ or $\hat{\bm\gamma}_{SROE}$.

The results are displayed in Figures 1-6. Figures 1-3 show the estimated SMSE values of the OLSE, RE, AURE, LE, AULE, PCRE, r-k class estimator, r-d class estimator and SIOE when  $\alpha=0.9$, $\alpha=0.99$ and $\alpha=0.999$ for the selected values of shrinkage parameters, respectively. Figures 4-6 show the estimated SMSE values of the MRE, SRRE, SRAURE, SRLE, SRAULE, SRPCRE, SRrk, SRrd and SROE when  $\alpha=0.9$, $\alpha=0.99$ and $\alpha=0.999$ for the selected values of shrinkage parameters, respectively.

\begin{figure}
\centering
\subfloat[Under correctly specified model.]{%
\resizebox*{2.78in}{!}{\includegraphics{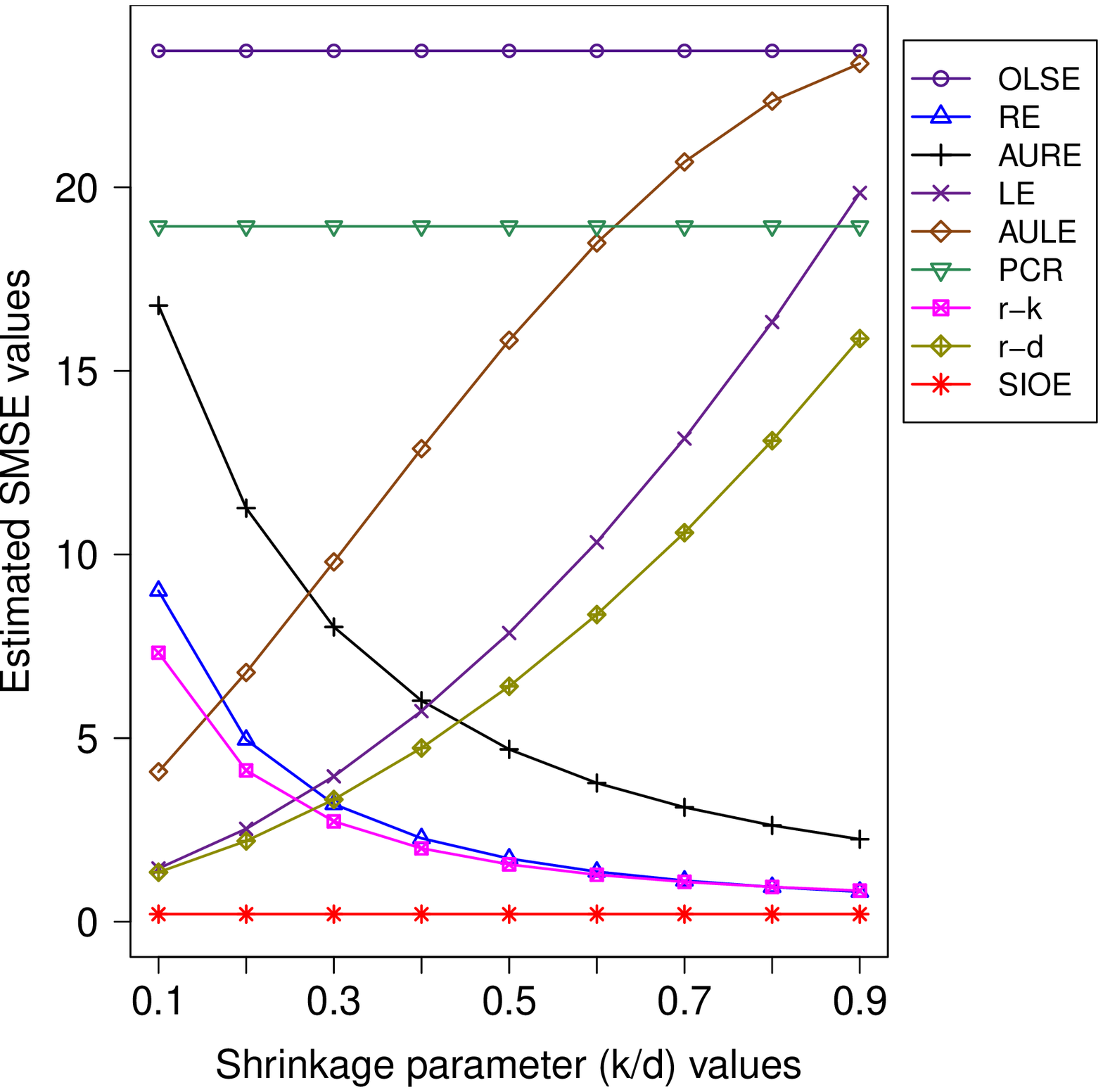}}}\hspace{2pt}
\subfloat[Under misspecified model.]{%
\resizebox*{2.78in}{!}{\includegraphics{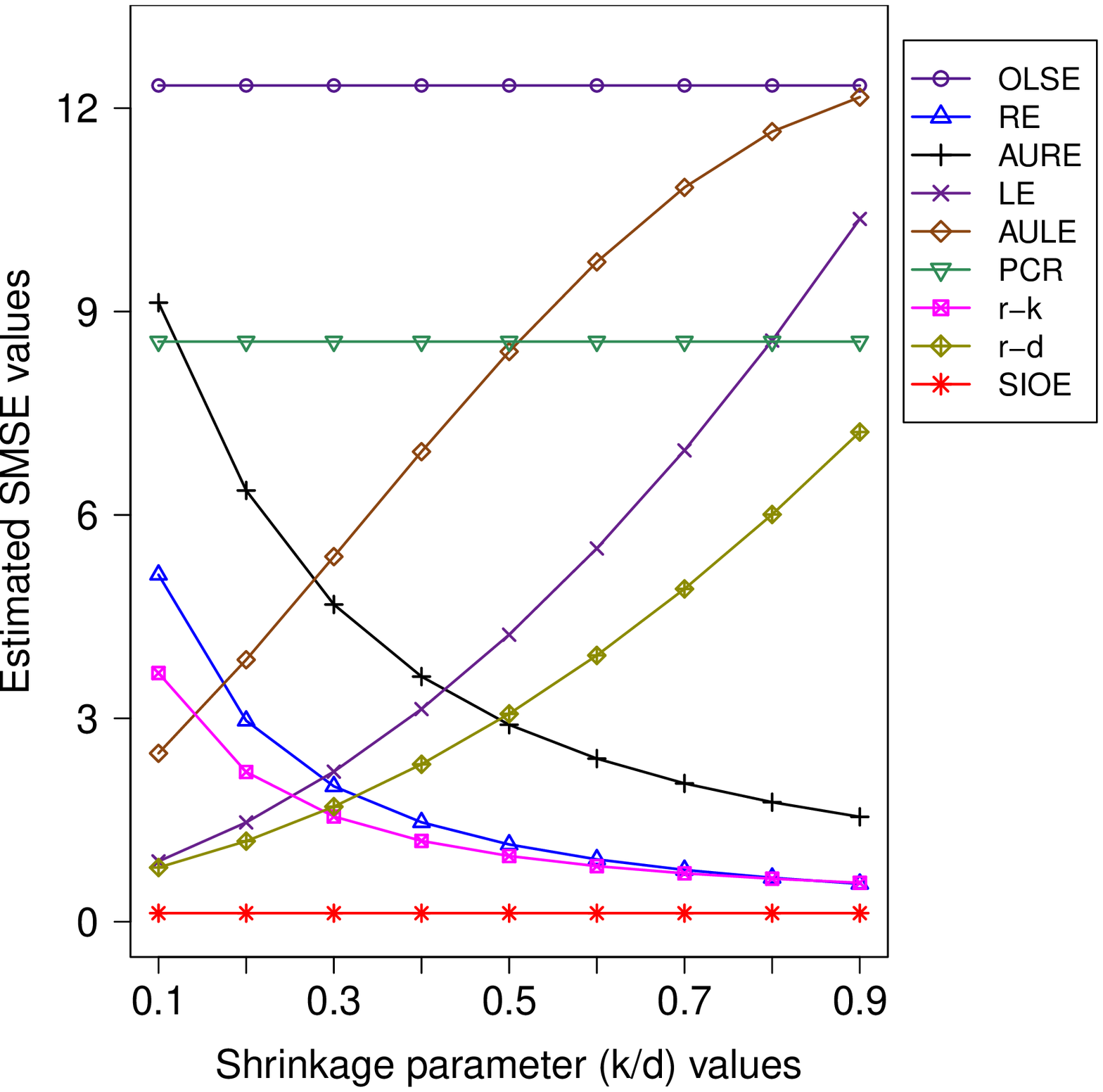}}}
\caption{SMSE values of the OLSE, RE, AURE, LE, AULE, PCRE, r-k class estimator, r-d class estimator and SIOE when $\alpha=0.9$.} \label{f1}
\end{figure}

\begin{figure}
\centering
\subfloat[Under correctly specified model.]{%
\resizebox*{2.78in}{!}{\includegraphics{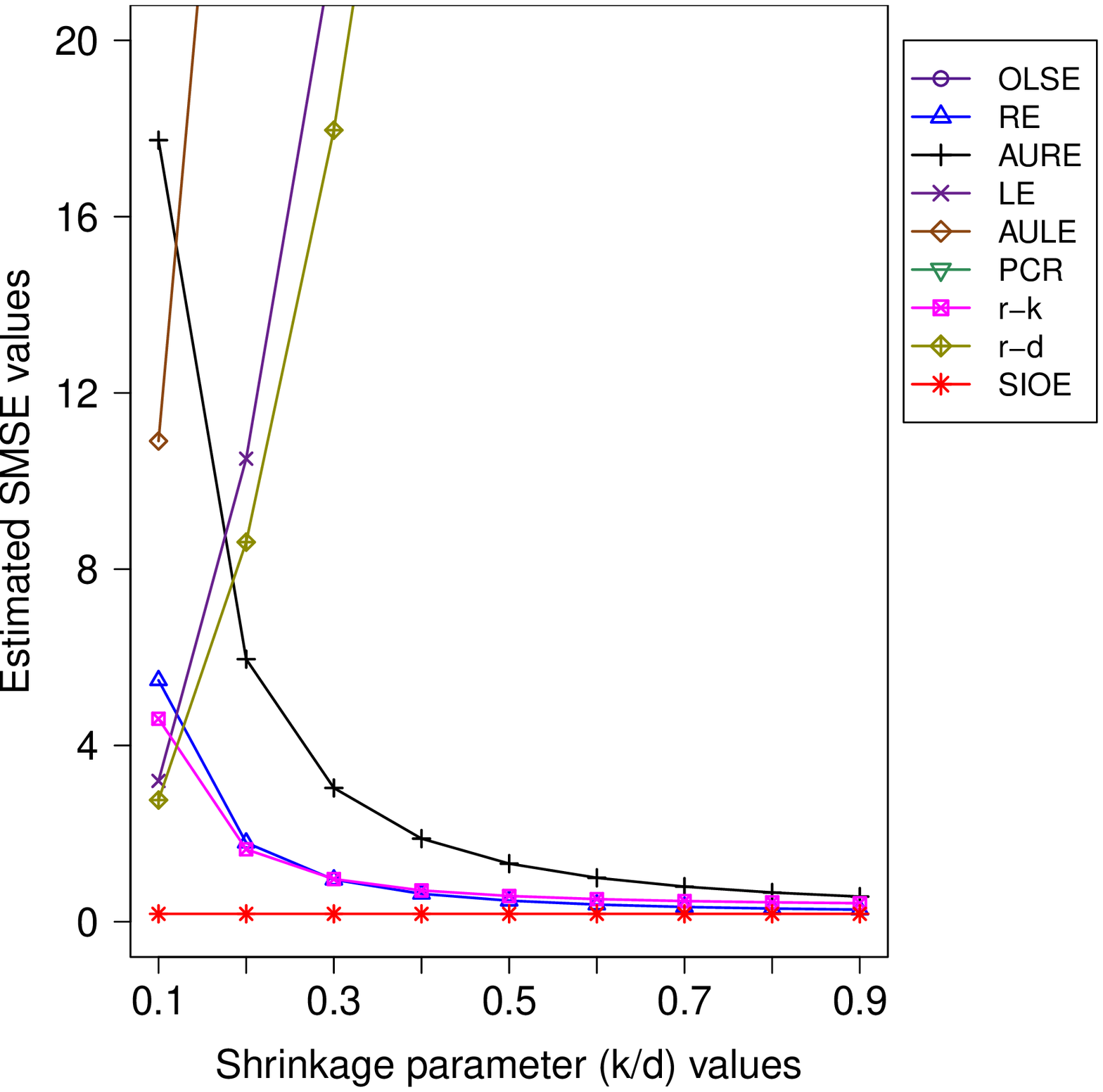}}}\hspace{2pt}
\subfloat[Under misspecified model.]{%
\resizebox*{2.78in}{!}{\includegraphics{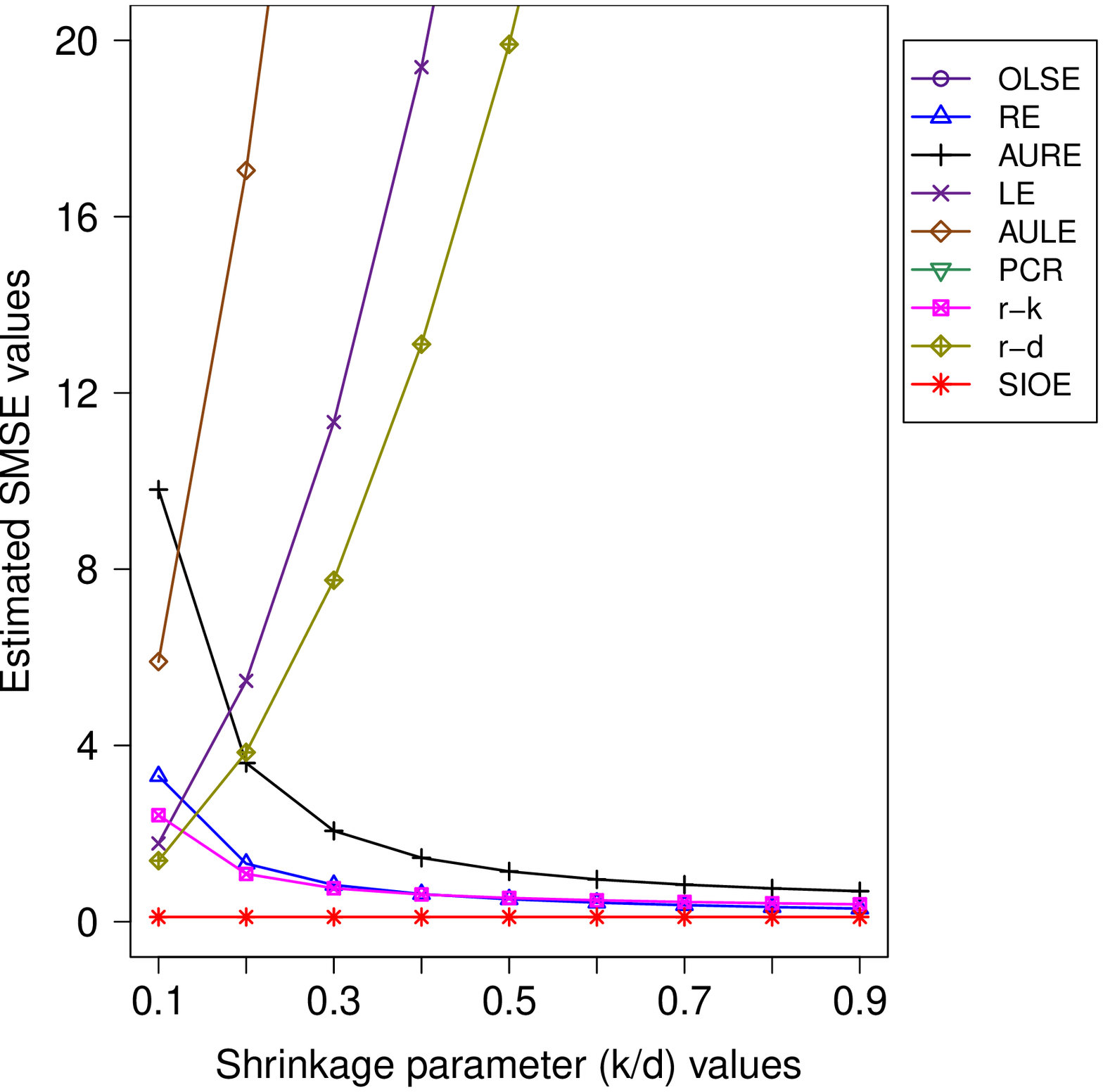}}}
\caption{SMSE values of the OLSE, RE, AURE, LE, AULE, PCRE, r-k class estimator, r-d class estimator and SIOE when $\alpha=0.99$ (the lines of OLSE, LE, AULE, PCRE and r-d class estimator are disappeared due to high SMSE values compared to others).} \label{f1}
\end{figure}
\begin{figure}
\centering
\subfloat[Under correctly specified model.]{%
\resizebox*{2.78in}{!}{\includegraphics{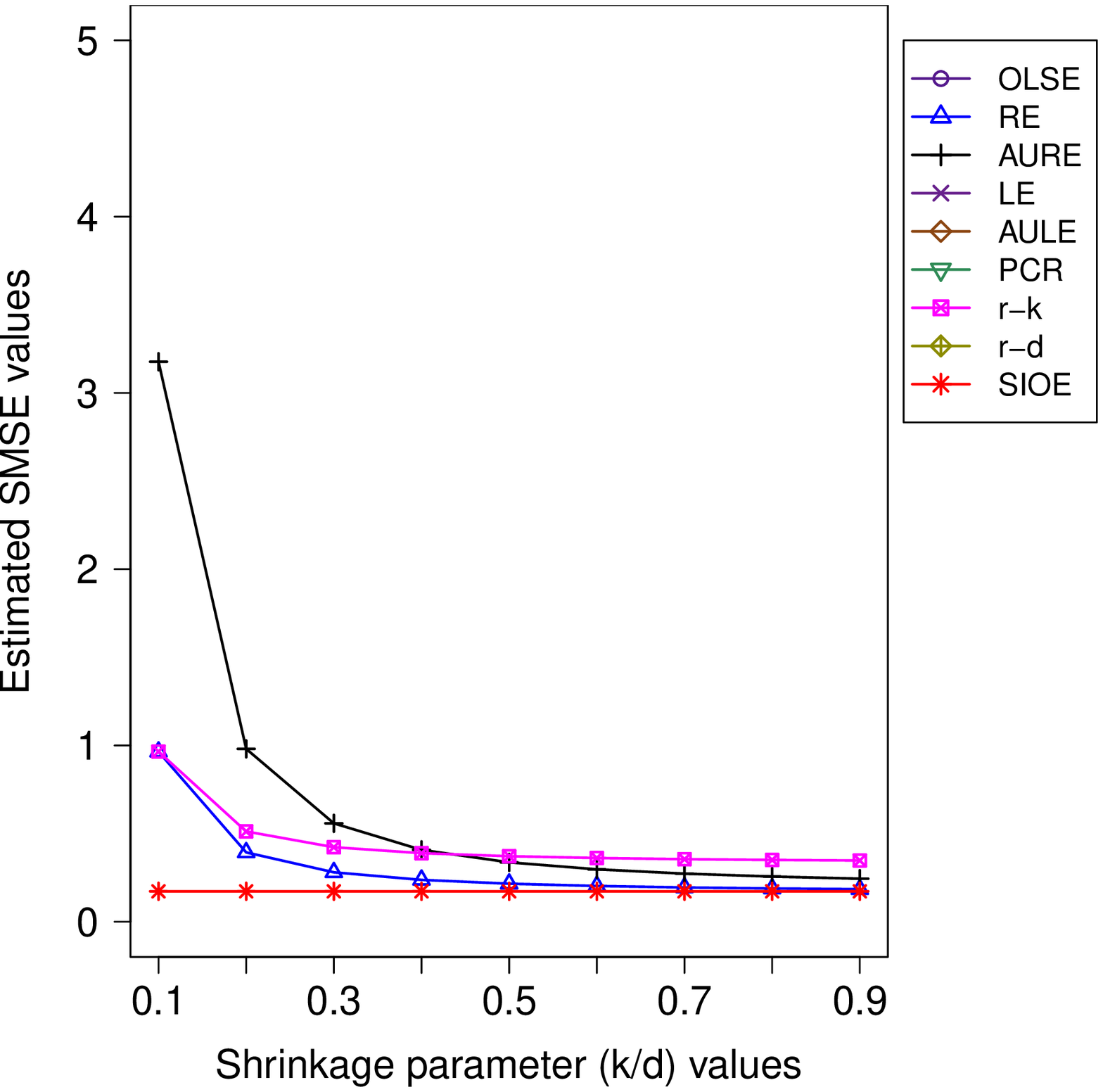}}}\hspace{2pt}
\subfloat[Under misspecified model.]{%
\resizebox*{2.78in}{!}{\includegraphics{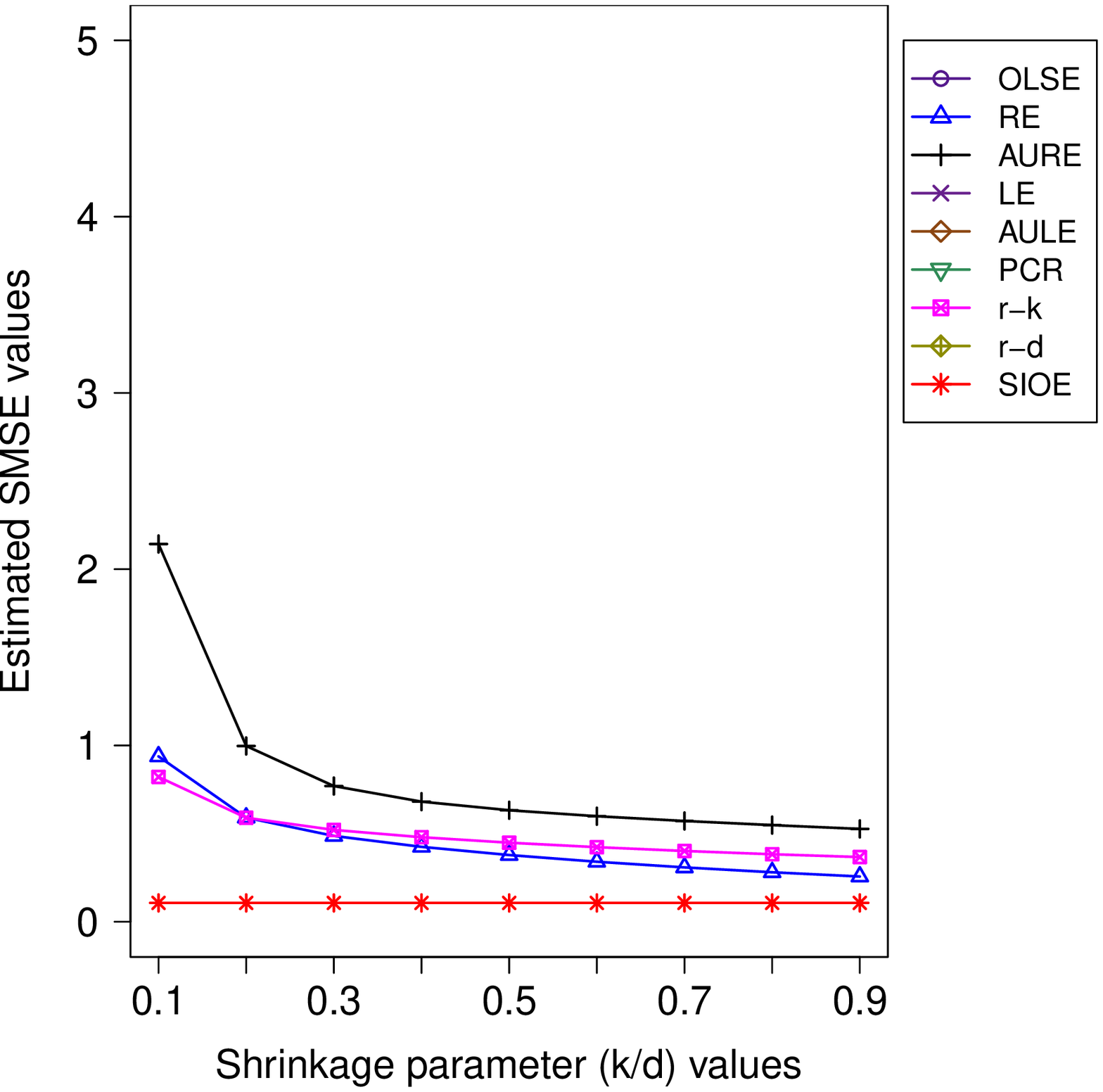}}}
\caption{SMSE values of the OLSE, RE, AURE, LE, AULE, PCRE, r-k class estimator, r-d class estimator and SIOE when $\alpha=0.999$ (the lines of OLSE, LE, AULE, PCRE and r-d class estimator are disappeared due to high SMSE values compared to others).} \label{f1}
\end{figure}

\begin{figure}
\centering
\subfloat[Under correctly specified model.]{%
\resizebox*{2.78in}{!}{\includegraphics{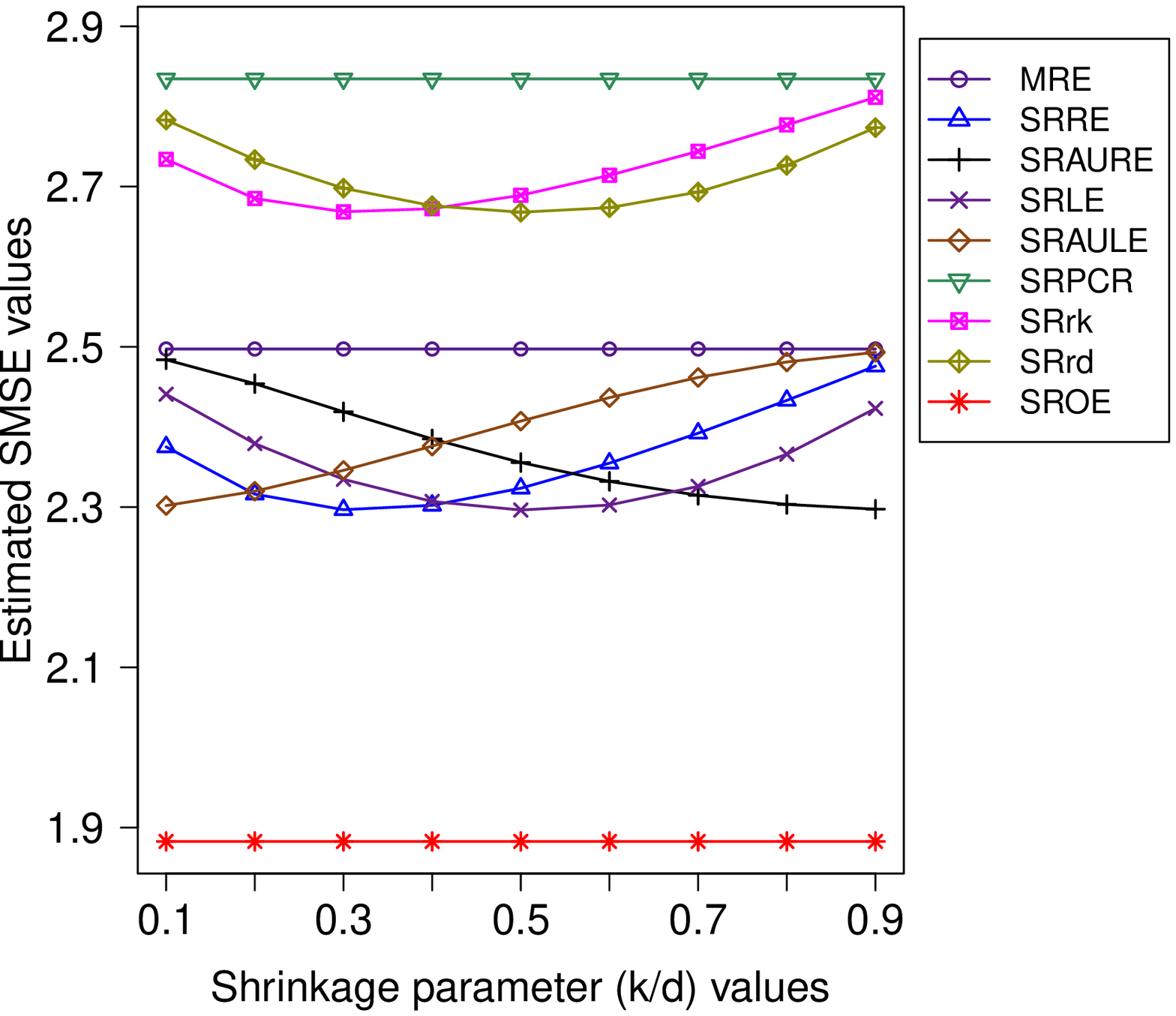}}}\hspace{2pt}
\subfloat[Under misspecified model.]{%
\resizebox*{2.78in}{!}{\includegraphics{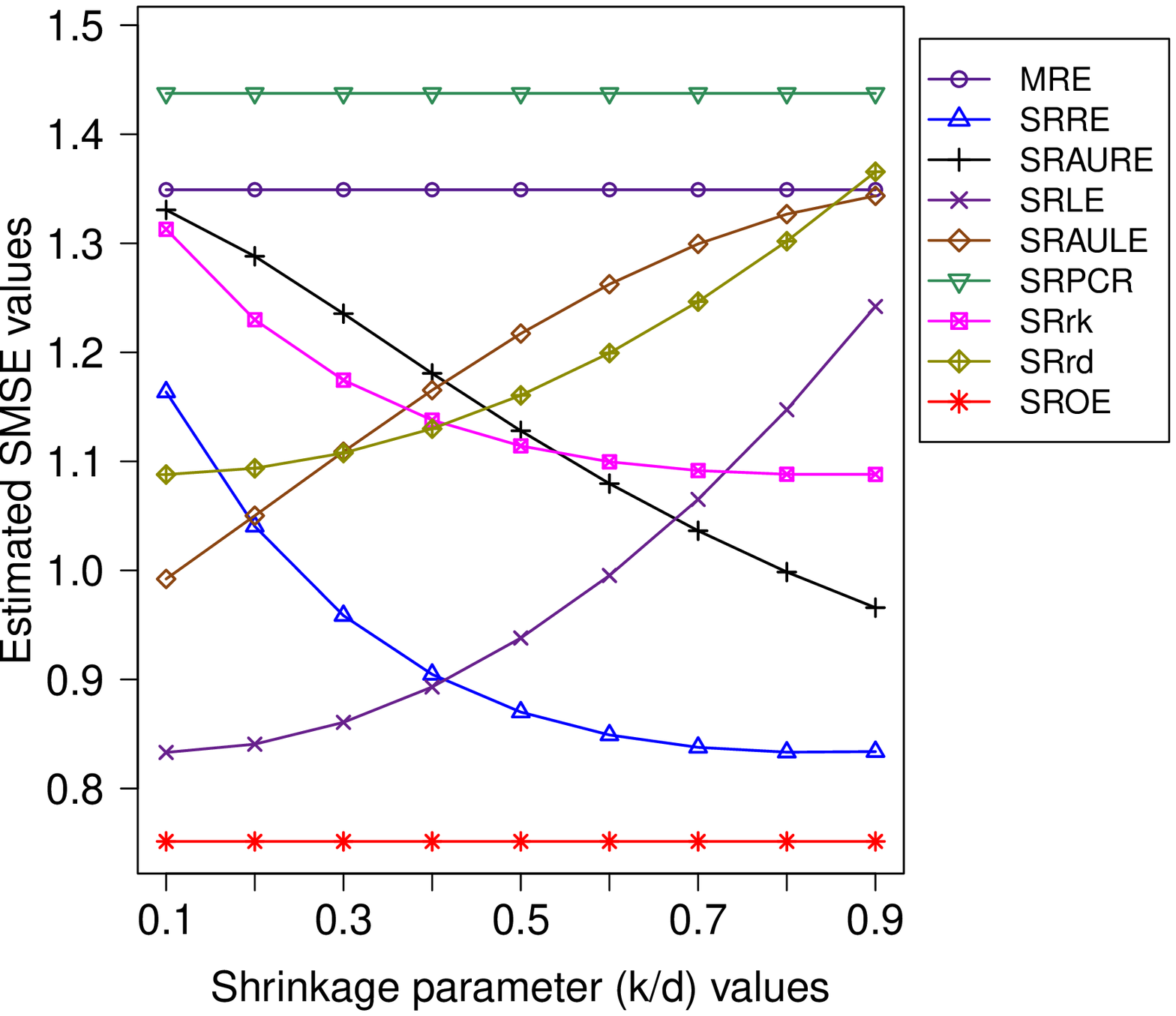}}}
\caption{SMSE values of the MRE, SRRE, SRAURE, SRLE, SRAULE, SRPCRE, SRrk, SRrd and SROE when $\alpha=0.9$.} \label{sample-figure}
\end{figure}

\begin{figure}
\centering
\subfloat[Under correctly specified model.]{%
\resizebox*{2.78in}{!}{\includegraphics{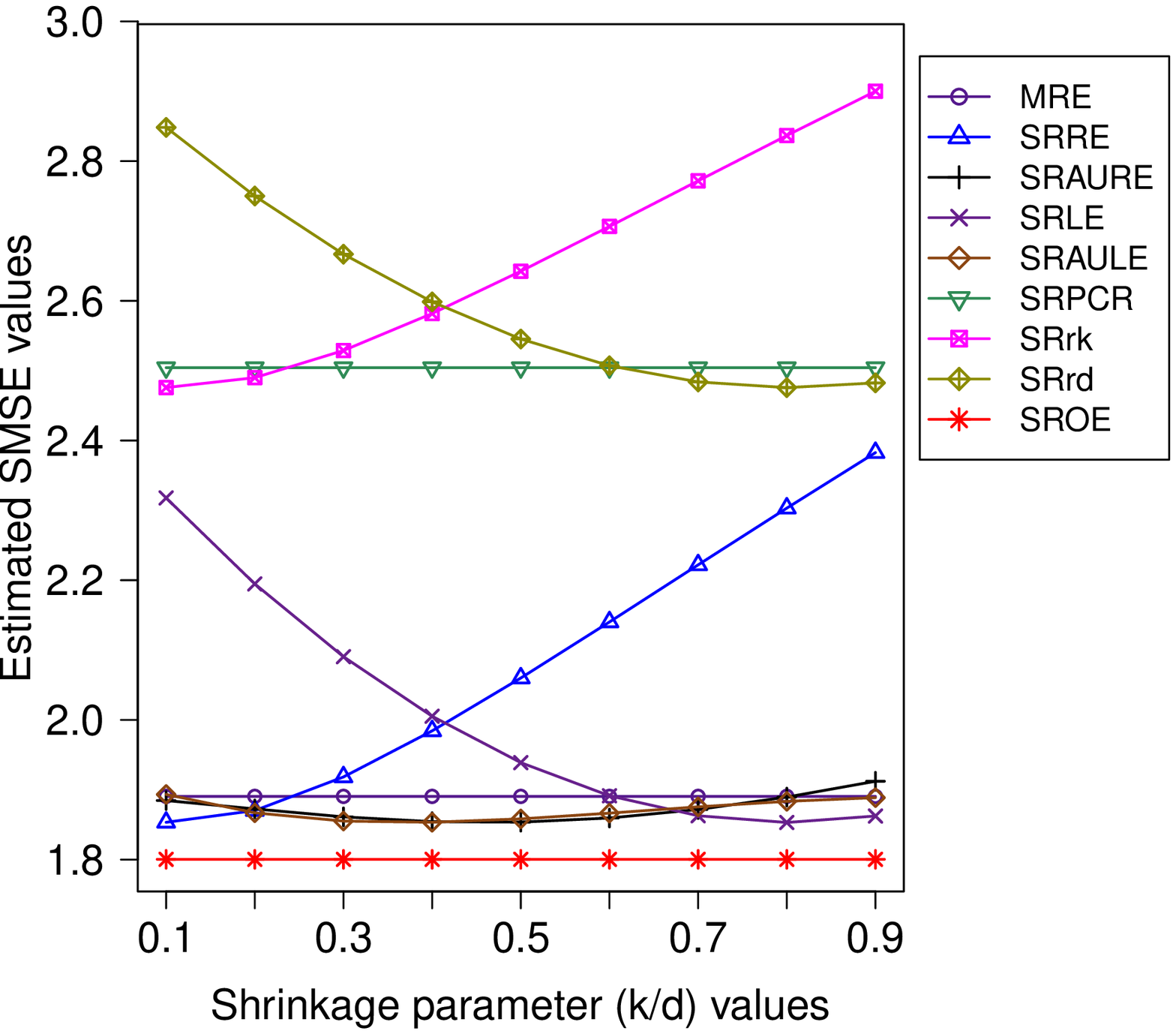}}}\hspace{2pt}
\subfloat[Under misspecified model.]{%
\resizebox*{2.78in}{!}{\includegraphics{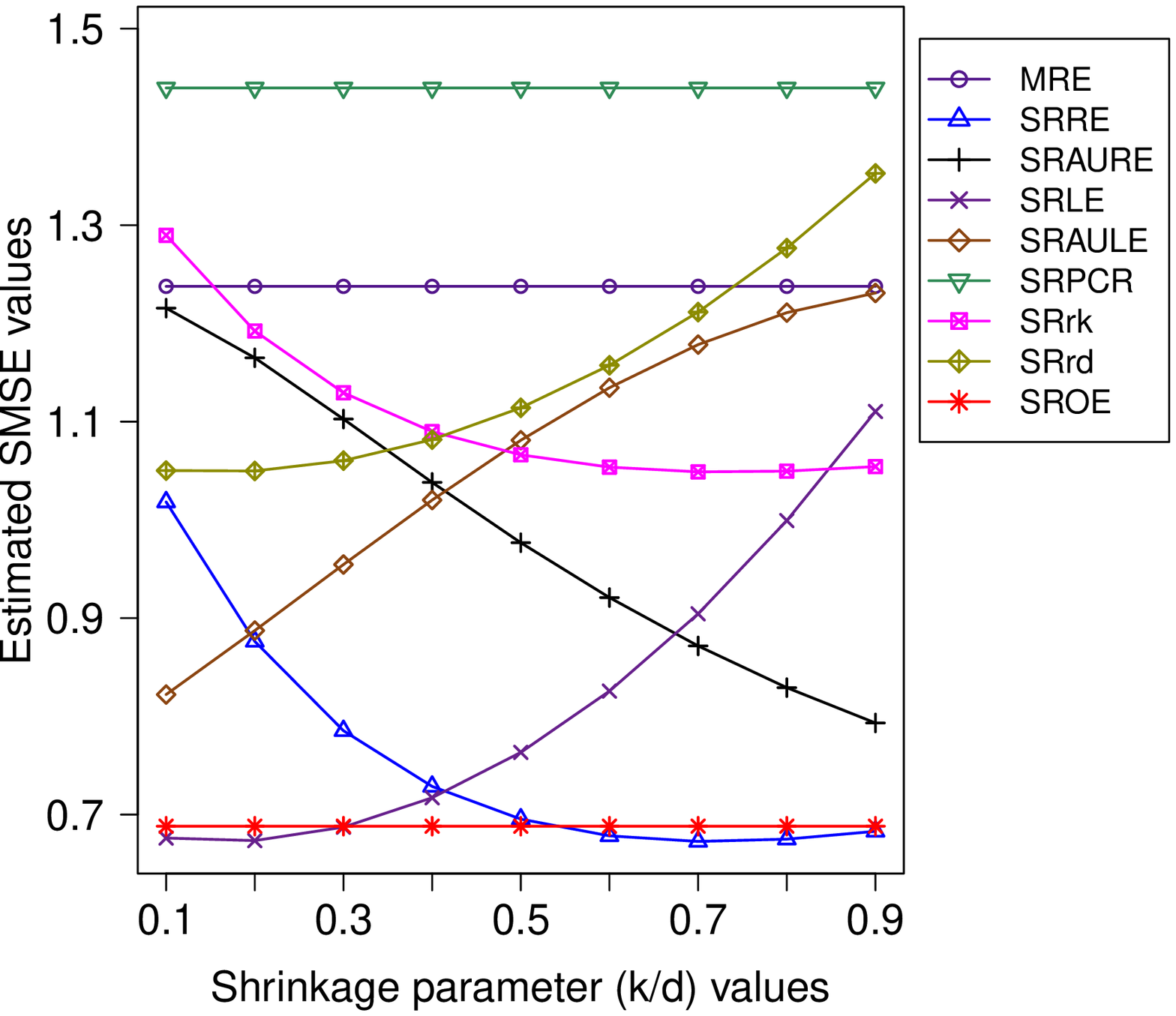}}}
\caption{SMSE values of the MRE, SRRE, SRAURE, SRLE, SRAULE, SRPCRE, SRrk, SRrd and SROE when $\alpha=0.99$.} \label{sample-figure}
\end{figure}

\begin{figure}
\centering
\subfloat[Under correctly specified model.]{%
\resizebox*{2.78in}{!}{\includegraphics{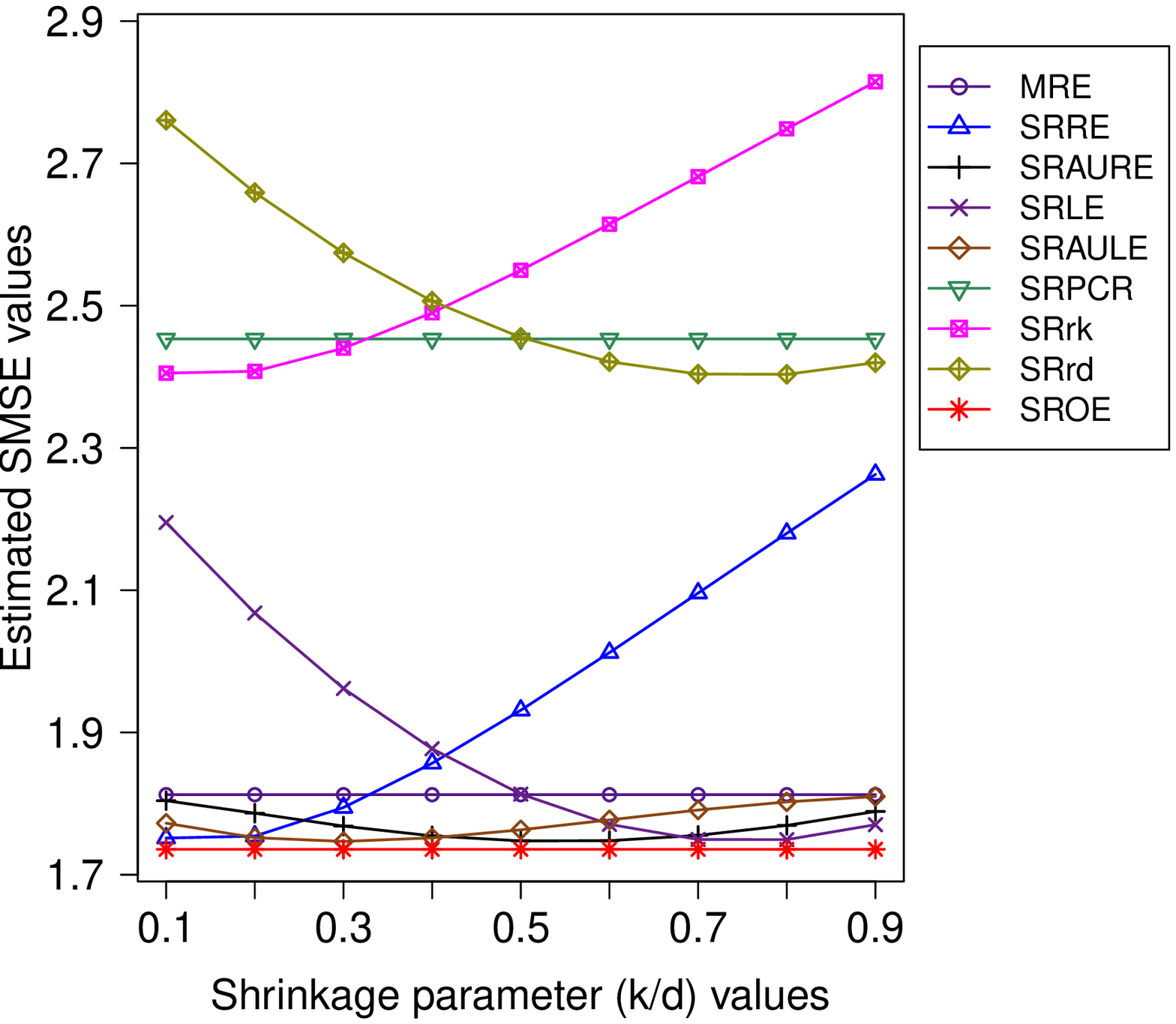}}}\hspace{2pt}
\subfloat[Under misspecified model.]{%
\resizebox*{2.78in}{!}{\includegraphics{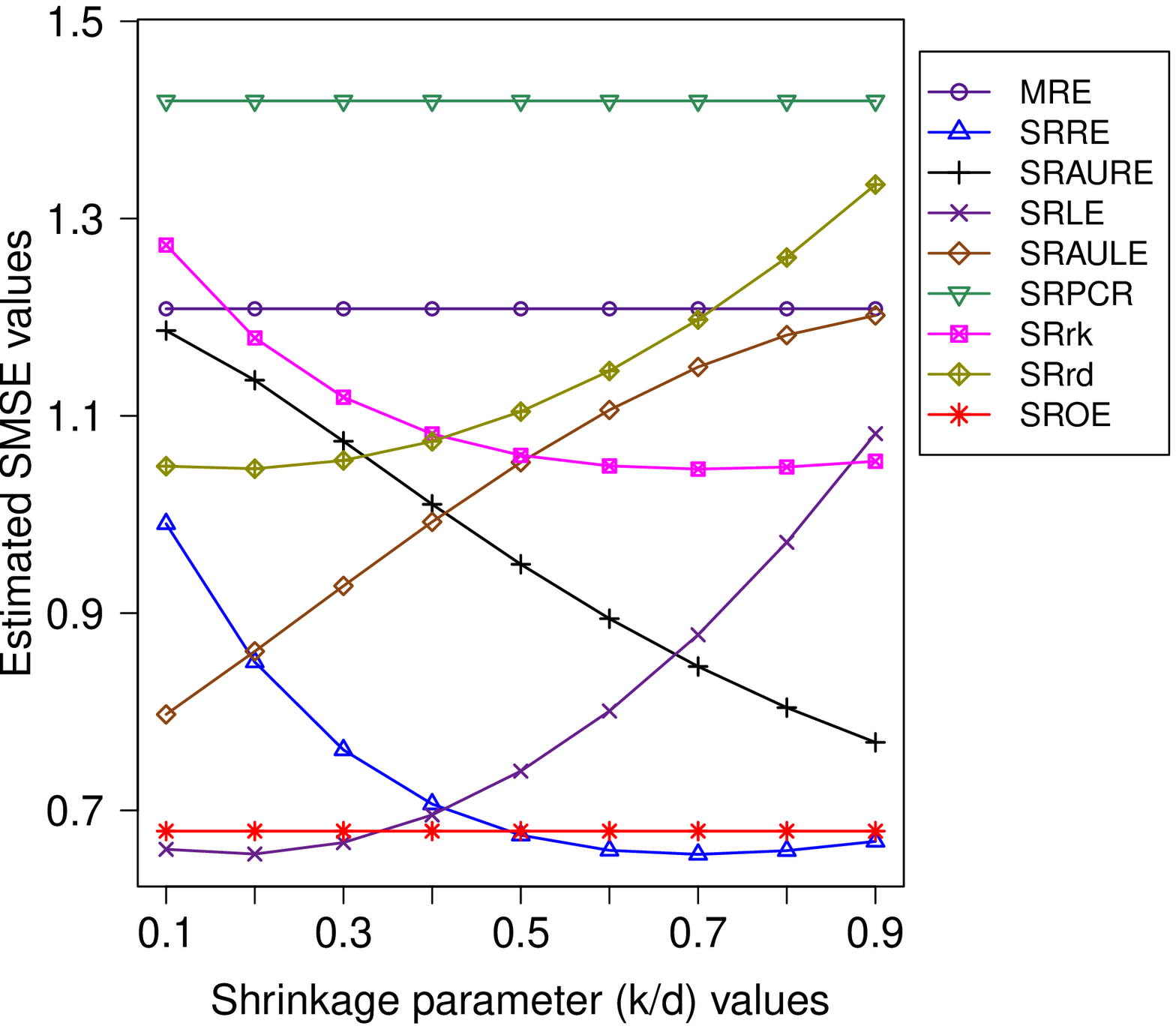}}}
\caption{SMSE values of the MRE, SRRE, SRAURE, SRLE, SRAULE, SRPCRE, SRrk, SRrd and SROE when $\alpha=0.999$.} \label{sample-figure}
\end{figure}

According to the results shown in Figures 1-3, we can observe that SIOE always outperforms the OLSE, RE, AURE, LE, AULE, PCRE, r-k class estimator, and r-d class estimator in both correctly specified model and misspecified model for all the given values of $\alpha$. According to the results shown in Figures 4-6, it is clear that SROE always outperforms the other estimators in the correctly specified model under all given values of $\alpha$. When we consider misspecified model\begin{enumerate}[(i)]
 \item SROE always outperforms over the MRE, SRRE, SRAURE, SRLE, SRAULE, SRPCRE, SRrk, and SRrd estimators when  $\alpha=0.9$, and
     \item SROE has approximately the same performance with SRRE and SRLE for a specific range of shrinkage parameter values when  $\alpha=0.99$ and $\alpha=0.999$.
     \end{enumerate}

  Further, this study shows that the performance of the estimators are different when the model is correctly specified and misspecified.
\subsection{Real-world example}
To illustrate the theoretical results, we further analyse the dataset \cite{Gr98} shown in Table C3 in Appendix C. This data set gives the total National Research and Development Expenditures as a Percent of Gross National Product by Country from 1972 to 1986. The dependent variable $y$ of this dataset is the percentage spent by the United States, and the regressor variables are $x_1$ is the percent spent by former Soviet Union, $x_2$ that spent by France, $x_3$ that spent by West Germany, and $x_4$ that spent by the Japan. The dataset has been analysed by Akdeniz and Erol \cite{Ak03}, Li and Yang \cite{Li10} and among others. They compared the SMSE of estimators based only on the correctly specified model. Therefore, this study based on the same dataset shows the consequences of the performance of estimators when the model is misspecified by excluding relevant variables.

The regressor variables mentioned above are multicollinear since the Variance Inflation Factors (VIF) of the regressor variables are 6.91, 21.58, 29.75, and 1.79.  Further, we consider $\bm R=(1,-2,-2,-2)$ and $\bm g=(0,0,0,0)$ by following Li and Yang \cite{Li10}. To study the effect of misspecification, we partition the regressor matrix as $\bm X_{1}=(x_{1},x_{2},x_{3})$ and $\bm X_{2}=(x_{4})$. The SMSE comparisions are displayed in Figures 7-8. Figure 7 shows the estimated SMSE values of the OLSE, RE, AURE, LE, AULE, PCRE, r-k class estimator, r-d class estimator and SIOE for the selected values of shrinkage parameters. Figure 8 shows the estimated SMSE values of the MRE, SRRE, SRAURE, SRLE, SRAULE, SRPCRE, SRrk, SRrd and SROE for the selected values of shrinkage parameters.
\begin{figure}
\centering
\subfloat[Under correctly specified model.]{%
\resizebox*{2.78in}{!}{\includegraphics{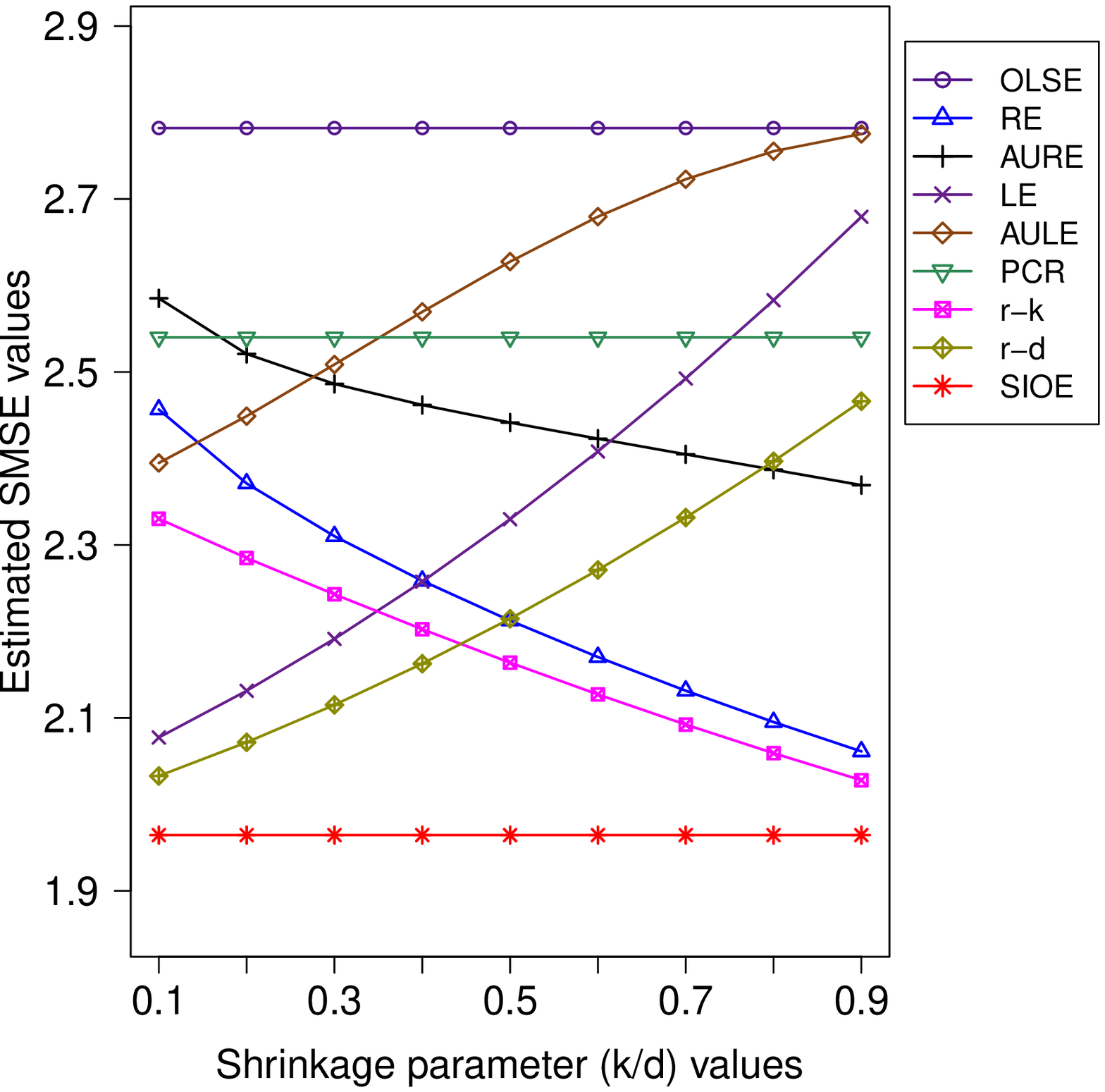}}}\hspace{2pt}
\subfloat[Under misspecified model.]{%
\resizebox*{2.78in}{!}{\includegraphics{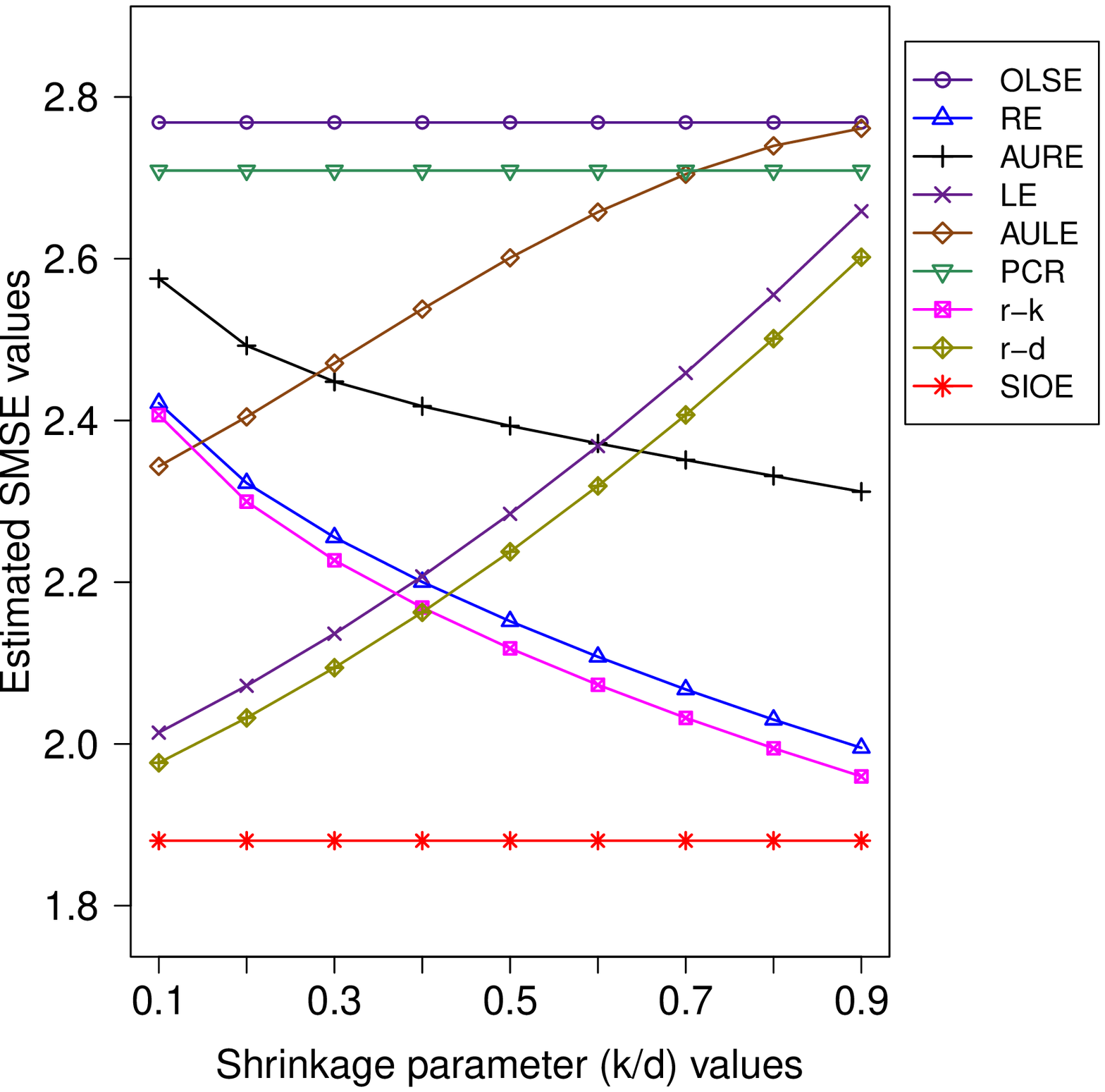}}}
\caption{SMSE values of the OLSE, RE, AURE, LE, AULE, PCRE, r-k class estimator, r-d class estimator and SIOE.} \label{f1}
\end{figure}

\begin{figure}
\centering
\subfloat[Under correctly specified model.]{%
\resizebox*{2.78in}{!}{\includegraphics{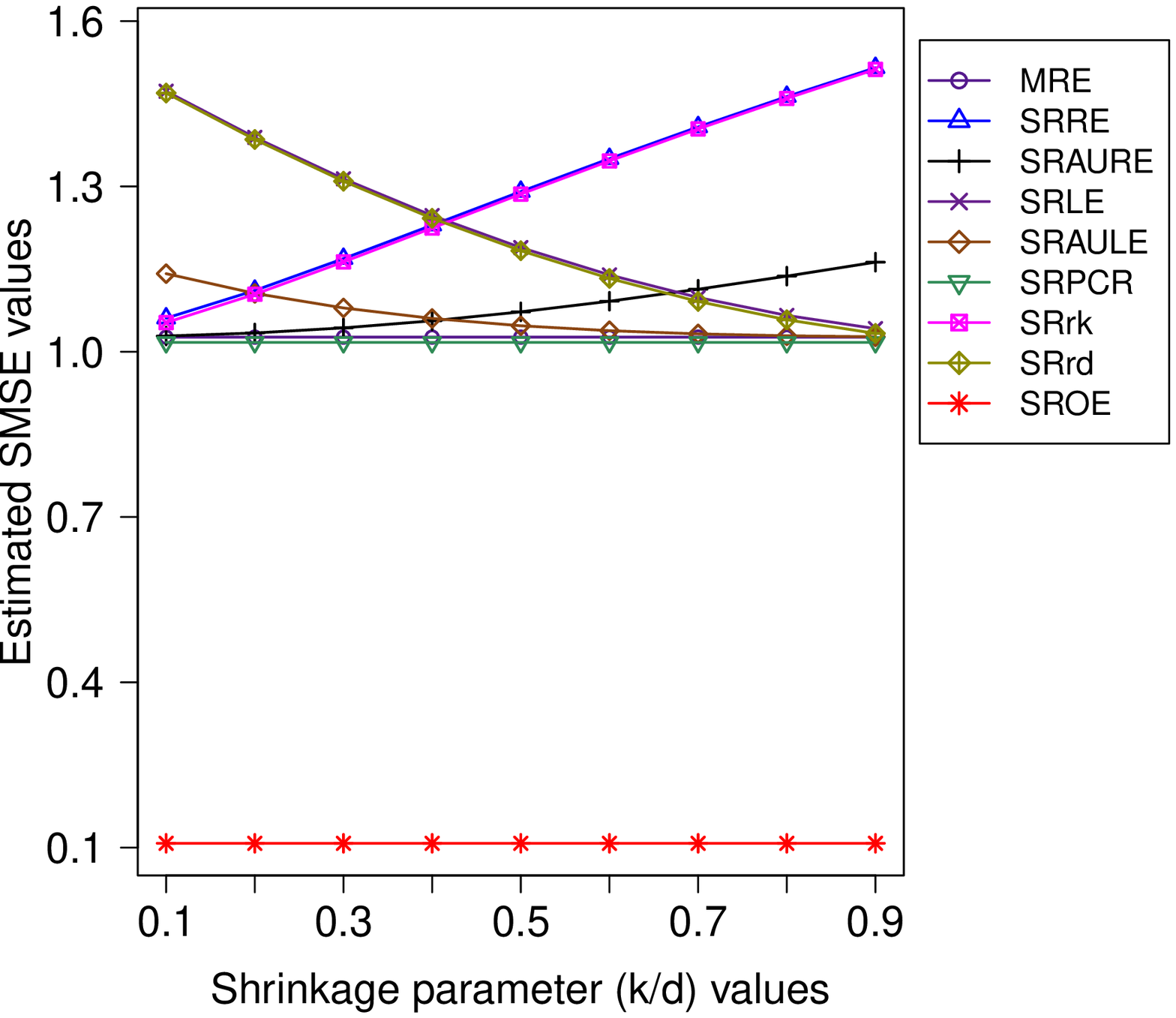}}}\hspace{2pt}
\subfloat[Under misspecified model.]{%
\resizebox*{2.78in}{!}{\includegraphics{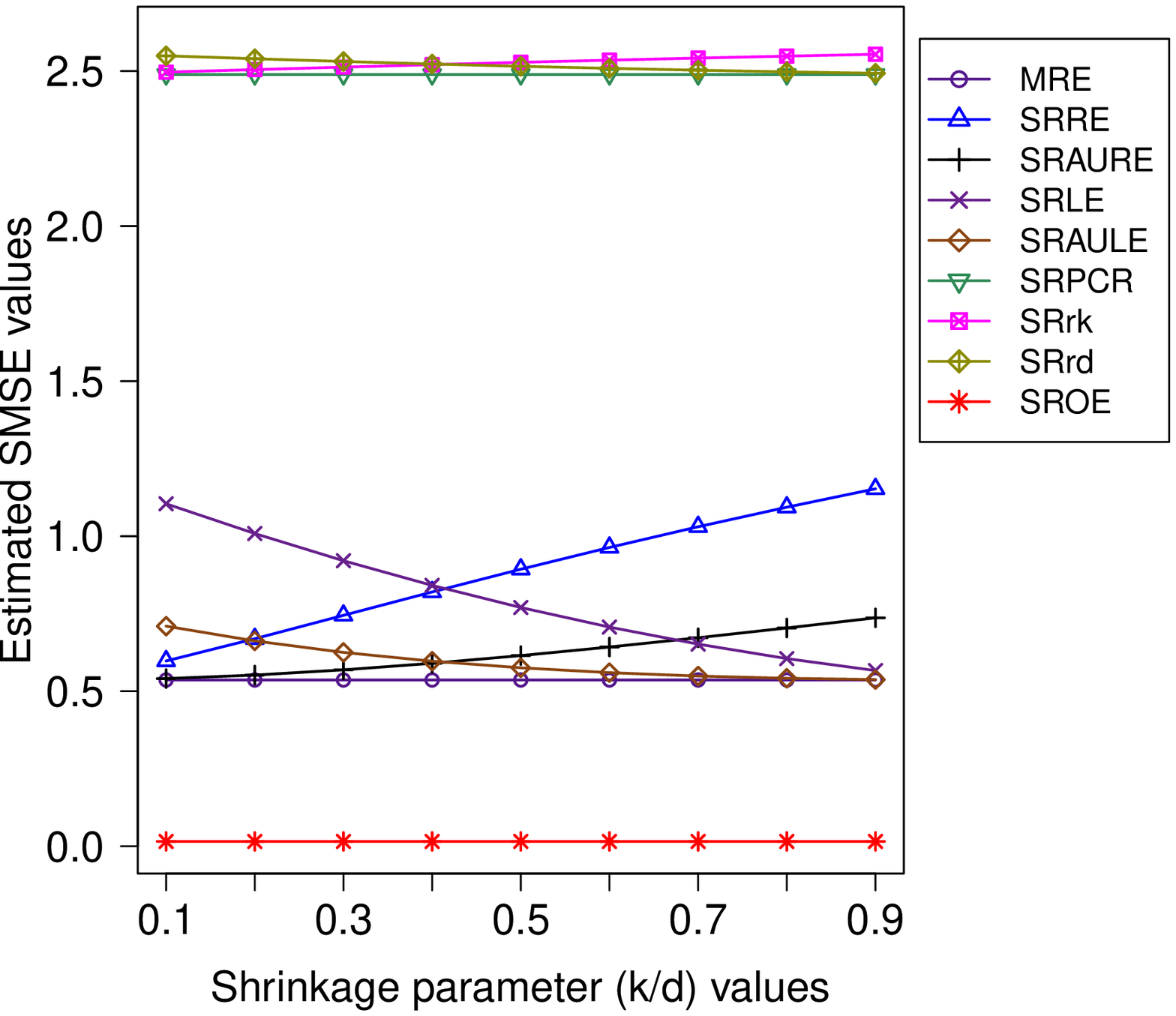}}}
\caption{SMSE values of the MRE, SRRE, SRAURE, SRLE, SRAULE, SRPCRE, SRrk, SRrd and SROE.} \label{sample-figure}
\end{figure}
 The results in Figure 7 demonstrate that SIOE always outperforms the OLSE, RE, AURE, LE, AULE, PCRE, r-k class estimator, and r-d class estimator in both correctly specified model and misspecified model. The results in Figure 8 demonstrate that SROE always outperforms the MRE, SRRE, SRAURE, SRLE, SRAULE, SRPCRE, SRrk, and SRrd estimators in both correctly specified model and misspecified model. However, it is evident that the SMSE comparisons show a significant difference when the model is correctly specified and misspecified.

\section{Conclusion}
 We showed that proposed optimal estimators SIOE and SROE are the best estimators when the model is correctly specified or misspecified regression model although multicollinearity exists among explanatory variables. Theoretically, we provided the superiority conditions for SIOE and SROE by  Theorem 3.1 and Theorem 3.2, respectively. Researchers can easily choose SIOE for the parameter estimation in either correctly specified regression model or misspecified regression model instead of considering several other biased estimators. Further, SROE can be used instead of SIOE if prior information is available on the regression coefficient $\bm \beta$. This study clearly shows the consequences of the performance of estimators when the regression model is misspecified by excluding relevant explanatory variables.

%
%\appendice
%\appendices
\section{Appendices}
\appendix
\section{Matrix Operations}
\begin{itemize}
  \item Let $M$ and $N$ be any two matrixes with proper order, then \[\frac{\partial\; tr(BMB')}{\partial B}=B(M+M')\]
\item Let $x$ is a $n \times 1$ vector, y is a $m\times 1$ vector and B is an $n\times m$ matrix, then\[\frac{\partial x'By}{\partial B}=xy'\]
\item Let $x$ be a $m \times 1$ vector, $M$ a $n \times n$ symmetric matrix and $B$ a $n \times m$ matrix, then\[\frac{\partial x'B'MBx}{\partial B}=2MBxx'\]
\end{itemize}
\section{Lemmas}
\begin{lemma}\label{section}
Let $\hat{\beta}_{1}$ and $\hat{\beta}_{2}$ be two linear estimator of $\beta$. Suppose that $D(\hat{\beta}_{1} )-D(\hat{\beta}_{2})$ is positive definite, then $MSEM(\hat{\beta}_{1})-MSEM(\hat{\beta}_{2})$ is non negative if and only if $b_{2}' \left(D(\hat{\beta}_{1} )-D(\hat{\beta}_{2})+b_{1} b_{1}'\right)^{-1} b_{2}\leq1$, where $D(\hat{\beta}_{j} )$,  $MSEM(\hat{\beta}_{j})$ and $b_{j}$ denote dispersion matrix, mean square error matrix and bias vector of $\hat{\beta}_{j}$ respectively, $j=1,2$ \cite{Tr90}.
\end{lemma}
\begin{lemma}\label{section}
Let $n\times n$ matrices $M > 0$, $N \geq 0$, then $M > N$ if and only if largest eigenvalue of the matrix $NM^{-1}$ is less than one \cite{ro95}.
\end{lemma}
\section{Tables}
\begin{table}[htb]
\tbl{Bias vector, Dispersion matrix and MSEM of the biased estimators }
{\begin{tabular}{cl}\toprule
Estimators & \multicolumn{ 1}{c}{Stochastic Properties}\\ \cmidrule{1-2}%$Bias(\hat{\gamma}_G )$, $D(\hat{\gamma}_G )$ and $MSEM(\hat{\gamma}_G ))$
\multirow{ 3}{*}{$\hat{\gamma}_{OLSE}$}& $Bias(\hat{\gamma}_{OLSE})=\tau A$\\
&$D(\hat{\gamma}_{OLSE})=\sigma^{2}\tau$\\
&$MSEM(\hat{\gamma}_{OLSE})=\sigma^{2}\tau+(\tau A)(\tau A)'$\\ \cmidrule{1-2}
\multirow{ 3}{*}{$\hat{\gamma}_{RE}$}&$Bias(\hat{\gamma}_{RE})=(\Lambda+kI)^{-1}(A-k\gamma)$\\
&$D(\hat{\gamma}_{RE})=\sigma^{2}(\Lambda+kI)^{-2}\Lambda$\\
&$MSEM(\hat{\gamma}_{RE})=(\Lambda+kI)^{-1}\{\sigma^{2}\Lambda+((A-k\gamma))((A-k\gamma))'\}(\Lambda+kI)^{-1}$\\ \cmidrule{1-2}
\multirow{ 3}{*}{$\hat{\gamma}_{AURE}$}&$Bias(\hat{\gamma}_{AURE})=(\Lambda+kI)^{-2}((\Lambda+2kI)A-k^{2}\gamma)$\\
&$D(\hat{\gamma}_{AURE})=\sigma^{2}(\Lambda+kI)^{-4}(\Lambda+2kI)^{2}\Lambda$\\
&$\begin{aligned}MSEM(\hat{\gamma}_{AURE})=&(\Lambda+kI)^{-2}\{\sigma^{2}(\Lambda+2kI)^{2}\Lambda\\
&+((\Lambda+2kI)A-k^{2}\gamma)((\Lambda+2kI)A-k^{2}\gamma)'\}(\Lambda+kI)^{-2}\end{aligned}$\\ \cmidrule{1-2}
\multirow{ 3}{*}{$\hat{\gamma}_{LE}$}& $Bias(\hat{\gamma}_{LE})=(\Lambda+I)^{-1}((I+d\tau)A-(1-d)\gamma)$\\
&$D(\hat{\gamma}_{LE})=\sigma^{2}(\Lambda+I)^{-2}(\Lambda+dI)^{2}\tau$\\
&$\begin{aligned}MSEM(\hat{\gamma}_{LE})=&(\Lambda+I)^{-1}\{\sigma^{2}(\Lambda+dI)^{2}\tau\\
&+((I+d\tau)A-(1-d)\gamma)((I+d\tau)A-(1-d)\gamma)'\}(\Lambda+I)^{-1}\end{aligned}$\\ \cmidrule{1-2}
\multirow{ 3}{*}{$\hat{\gamma}_{AULE}$}& $Bias(\hat{\gamma}_{AULE})=(\Lambda+I)^{-2}((\Lambda+(2-d)I)(I+d\tau)A-(1-d)^{2}\gamma)$\\
&$D(\hat{\gamma}_{AULE})=\sigma^{2}(\Lambda+I)^{-4}(\Lambda+dI)^{2}(\Lambda+(2-d)I)^{2}\tau$\\
&$\begin{aligned}MSEM(\hat{\gamma}_{AULE})=&(\Lambda+I)^{-2}\{\sigma^{2}(\Lambda+dI)^{2}(\Lambda+(2-d)I)^{2}\tau+((\Lambda+(2-d)I)(I+d\tau)A\\
&-(1-d)^{2}\gamma)((\Lambda+(2-d)I)(I+d\tau)A-(1-d)^{2}\gamma)'\}(\Lambda+I)^{-2}\end{aligned}$\\ \cmidrule{1-2}
\multirow{ 3}{*}{$\hat{\gamma}_{PCRE}$}& $Bias(\hat{\gamma}_{PCRE})=(T_{h}T_{h}'-I)\gamma+T_{h}T_{h}'\tau A$\\
&$D(\hat{\gamma}_{PCRE})=\sigma^{2}T_{h}T_{h}'\tau T_{h}'T_{h}$\\
&$MSEM(\hat{\gamma}_{PCRE})=\sigma^{2}T_{h}T_{h}'\tau T_{h}'T_{h}+\left((T_{h}T_{h}'-I)\gamma+T_{h}T_{h}'\tau A\right)\left((T_{h}T_{h}'-I)\gamma+T_{h}T_{h}'\tau A\right)'$\\ \cmidrule{1-2}
\multirow{ 3}{*}{$\hat{\gamma}_{rk}$}& $Bias(\hat{\gamma}_{rk})=(T_{h}T_{h}'(\Lambda+kI)^{-1}\Lambda-I)\gamma+T_{h}T_{h}'(\Lambda+kI)^{-1} A$\\
&$D(\hat{\gamma}_{rk})=\sigma^{2}T_{h}T_{h}'(\Lambda+kI)^{-2}\Lambda T_{h}'T_{h}$\\
&$\begin{aligned}MSEM(\hat{\gamma}_{rk})=&\sigma^{2}T_{h}T_{h}'(\Lambda+kI)^{-2}\Lambda T_{h}'T_{h}+((T_{h}T_{h}'(\Lambda+kI)^{-1}\Lambda-I)\gamma\\
&+T_{h}T_{h}'(\Lambda+kI)^{-1} A)((T_{h}T_{h}'(\Lambda+kI)^{-1}\Lambda-I)\gamma+T_{h}T_{h}'(\Lambda+kI)^{-1} A)'\end{aligned}$\\ \cmidrule{1-2}
\multirow{ 3}{*}{$\hat{\gamma}_{rd}$}& $Bias(\hat{\gamma}_{rd})=(T_{h}T_{h}'(\Lambda+I)^{-1}(\Lambda+dI)-I)\gamma+T_{h}T_{h}'(\Lambda+I)^{-1}(I+d\tau) A$\\
&$D(\hat{\gamma}_{rd})=\sigma^{2}T_{h}T_{h}'(\Lambda+I)^{-2}(\Lambda+dI)^{2}\tau T_{h}'T_{h}$\\
&$\begin{aligned}MSEM(\hat{\gamma}_{rd})=&\sigma^{2}T_{h}T_{h}'(\Lambda+I)^{-2}(\Lambda+dI)^{2}\tau T_{h}'T_{h}+((T_{h}T_{h}'(\Lambda+I)^{-1}(\Lambda+dI)-I)\gamma+T_{h}T_{h}'(\Lambda\\
&+I)^{-1}(I+d\tau) A)((T_{h}T_{h}'(\Lambda+I)^{-1}(\Lambda+dI)-I)\gamma+T_{h}T_{h}'(\Lambda+I)^{-1}(I+d\tau) A)'\end{aligned}$\\ \bottomrule
\end{tabular}}
\label{section}
\end{table}

\begin{table}
\tbl{Bias vector, Dispersion matrix and MSEM of the stochastic restricted estimators }
{\begin{tabular}{cl}\toprule
Estimators & \multicolumn{ 1}{c}{Stochastic Properties}\\ \midrule%$Bias(\hat{\gamma}_G )$, $D(\hat{\gamma}_G )$ and $MSEM(\hat{\gamma}_G ))$
\multirow{ 3}{*}{$\hat{\gamma}_{MRE}$}& $Bias(\hat{\gamma}_{MRE})=\tau A$\\
&$D(\hat{\gamma}_{MRE})=\sigma^{2}\tau$\\
&$MSEM(\hat{\gamma}_{MRE})=\sigma^{2}\tau+(\tau A)(\tau A)'$\\ \midrule
\multirow{ 3}{*}{$\hat{\gamma}_{SRRE}$}&$Bias(\hat{\gamma}_{SRRE})=(1+k)^{-1}(\tau A-k\gamma)$\\
&$D(\hat{\gamma}_{SRRE})=(1+k)^{-2}\sigma^{2}\tau$\\
&$MSEM(\hat{\gamma}_{SRRE})=(1+k)^{-2}\left(\sigma^{2}\tau+(\tau A-k\gamma)(\tau A-k\gamma)'\right)$\\ \midrule
\multirow{ 3}{*}{$\hat{\gamma}_{SAURRE}$}&$Bias(\hat{\gamma}_{SRAURE})=(1+k)^{-2}((1+2k)\tau A-k^{2}\gamma)$\\
&$D(\hat{\gamma}_{SRAURE})=(1+k)^{-4}(1+2k)^{2}\sigma^{2}\tau$\\
&$MSEM(\hat{\gamma}_{SRAURE})=(1+k)^{-4}\left((1+2k)^{2}\sigma^{2}\tau+((1+2k)\tau A-k^{2}\gamma)((1+2k)\tau A-k^{2}\gamma)'\right)$\\ \midrule
\multirow{ 3}{*}{$\hat{\gamma}_{SRLE}$}&$Bias(\hat{\gamma}_{SRLE})=2^{-1}((1+d)\tau A-(1-d)\gamma)$\\
&$D(\hat{\gamma}_{SRLE})=2^{-2}(1+d)^{2}\sigma^{2}\tau$\\
&$MSEM(\hat{\gamma}_{SRLE})=2^{-2}\left((1+d)^{2}\sigma^{2}\tau+((1+d)\tau A-(1-d)\gamma)((1+d)\tau A-(1-d)\gamma)'\right)$\\ \midrule
\multirow{ 3}{*}{$\hat{\gamma}_{SRAULE}$}&$Bias(\hat{\gamma}_{SRAULE})=2^{-2}((1+d)(3-d)\tau A-(1-d)^{2}\gamma)$\\
&$D(\hat{\gamma}_{SRAULE})=2^{-4}(1+d)^{2}(3-d)^{2}\sigma^{2}\tau$\\
&$\begin{aligned}MSEM(\hat{\gamma}_{SRAULE})=&2^{-4}\{(1+d)^{2}(3-d)^{2}\sigma^{2}\tau\\
&+((1+d)(3-d)\tau A-(1-d)^{2}\gamma)((1+d)(3-d)\tau A-(1-d)^{2}\gamma)'\}\end{aligned}$\\ \midrule
\multirow{ 3}{*}{$\hat{\gamma}_{SRPCRE}$}& $Bias(\hat{\gamma}_{SRPCRE})=(T_{h}T_{h}'-I)\gamma+T_{h}T_{h}'\tau A$\\
&$D(\hat{\gamma}_{SRPCRE})=\sigma^{2}T_{h}T_{h}'\tau T_{h}'T_{h}$\\
&$MSEM(\hat{\gamma}_{SRPCRE})=\sigma^{2}T_{h}T_{h}'\tau T_{h}'T_{h}+\left((T_{h}T_{h}'-I)\gamma+T_{h}T_{h}'\tau A\right)\left((T_{h}T_{h}'-I)\gamma+T_{h}T_{h}'\tau A\right)'$\\ \midrule
\multirow{ 3}{*}{$\hat{\gamma}_{SRrk}$}& $Bias(\hat{\gamma}_{SRrk})=(1+k)^{-1}((T_{h}T_{h}'-(1+k)I)\gamma+T_{h}T_{h}'\tau A)$\\
&$D(\hat{\gamma}_{SRrk})=(1+k)^{-2}\sigma^{2}T_{h}T_{h}'\tau T_{h}'T_{h}$\\
&$\begin{aligned}MSEM(\hat{\gamma}_{SRrk})=&(1+k)^{-2}\{\sigma^{2}T_{h}T_{h}'\tau T_{h}'T_{h}\\
&+\left((T_{h}T_{h}'-(1+k)I)\gamma+T_{h}T_{h}'\tau A\right)\left((T_{h}T_{h}'-(1+k)I)\gamma+T_{h}T_{h}'\tau A\right)'\}\end{aligned}$\\ \midrule
\multirow{ 3}{*}{$\hat{\gamma}_{SRrd}$}& $Bias(\hat{\gamma}_{SRrd})=2^{-1}(1+d)((T_{h}T_{h}'-2(1+d)^{-1}I)\gamma+T_{h}T_{h}'\tau A)$\\
&$D(\hat{\gamma}_{SRrd})=2^{-2}(1+d)^{2}\sigma^{2}T_{h}T_{h}'\tau T_{h}'T_{h}$\\
&$\begin{aligned}MSEM(\hat{\gamma}_{SRrd})=&2^{-2}(1+d)^{2}\{\sigma^{2}T_{h}T_{h}'\tau T_{h}'T_{h}\\
&+\left((T_{h}T_{h}'-2(1+d)^{-1}I)\gamma+T_{h}T_{h}'\tau A\right)\left((T_{h}T_{h}'-2(1+d)^{-1}I)\gamma+T_{h}T_{h}'\tau A\right)'\}\end{aligned}$\\ \bottomrule
\end{tabular}}
\label{section}
\end{table}

\begin{table}
\tbl{Total National Research and Development Expenditures—as a Percent of Gross National Product by Country: 1972-1986 }
{\begin{tabular}{cccccc}\toprule
Year &$y$ &$x_{1}$ &$x_{2}$ &$x_{3}$& $x_{4}$\\ \midrule
1972 &2.3& 1.9& 2.2& 1.9 &3.7\\
1975 &2.2& 1.8& 2.2& 2.0& 3.8\\
1979 &2.2& 1.8& 2.4& 2.1& 3.6\\
1980 &2.3& 1.8& 2.4& 2.2& 3.8\\
1981 &2.4& 2.0& 2.5& 2.3& 3.8\\
1982 &2.5& 2.1& 2.6& 2.4& 3.7\\
1983 &2.6& 2.1& 2.6& 2.6& 3.8\\
1984 &2.6& 2.2& 2.6& 2.6& 4.0\\
1985 &2.7& 2.3& 2.8& 2.8& 3.7\\
1986& 2.7& 2.3& 2.7& 2.8& 3.8\\ \bottomrule
\end{tabular}}
\label{section}
\end{table}


\begin{thebibliography}{99}

\bibitem{Ho70}%1
Hoerl~A, Kennard~R. Ridge regression: Biased estimation for nonorthogonal problems. Technometrics. 1970;12:55-–67.
\bibitem{Si86}%2
Singh~B, Chaubey~YP, Dwivedi~TD. An Almost Unbiased Ridge Estimator. The Indian Journal of Statistics. 1986;48(3):342--346.
\bibitem{Lu93}%3
Liu~K. A new class of biased estimate in linear regression. Communication in Statistics - Theory and Methods. 1993;22:393-–402.
\bibitem{Ak95}%4
Akdeniz~F, Kaçiranlar~S. On the almost unbiased generalized liu estimator and unbiased estimation of the bias and mse. Communications in Statistics - Theory and Methods. 1995;24(7):1789--1797.
\bibitem{Ma65}%5
Massy~F. Principal components regression in exploratory statistical research. Journal of the American Statistical Association. 1965;60:234--266.
\bibitem{Ba84}%6
Baye~R, Parker~F. Combining ridge and principal component regression: A money demand illustration. Communications in Statistics-Theory and Methods. 1984;13(2):197--205.
\bibitem{Ka01}%7
Ka\c{c}iranlar~S,  Sakallıo\u{g}lu~S. Combining the Liu estimator and the principal component. Communications in Statistics-Theory and Methods. 2001;30(12):2699--2705.
\bibitem{th61}%8
Theil~H, Goldberger~AS. On Pure and Mixed Statistical Estimation in Economics. International Economic Review. 1961;2(1):65-78. %doi:10.2307/2525589
\bibitem{Hu04}%9
Hubert~M, Wijekoon~P. Superiority of the stochastic restricted Liu estimator under misspecification. Statistica. 2004;64(1):153--162. %doi:10.6092/issn.1973-2201/29
\bibitem{Li10}%10
Li~Y, Yang~H. A new stochastic mixed ridge estimator in linear regression model. Statistical Papers. 2010;315–-323.
%doi:10.1007/s00362-008-0169-5.
\bibitem{Ji14}%11
Jibo~W, Hu~Y. On the Stochastic Restricted Almost Unbiased Estimators in Linear Regression Model. Communications in Statistics - Simulation and Computation. 2014;428-440.
%doi:10.1080/03610918.2012.704540
\bibitem{He14}%12
He~D, Wu Y. A Stochastic Restricted Principal Components Regression Estimator in the Linear Model. The Scientific World Journal. 2014.%doi:10.1155/2014/231506
\bibitem{Jib14}%13
Jibo~W. On the Stochastic Restricted r-k Class Estimator and Stochastic Restricted r-d Class Estimator in Linear Regression Model. Journal of Applied Mathematics. 2014. %doi:10.1155/2014/173836
\bibitem{Sa89}%14
Sarkar~N. Comparisons Among Some Estimators In Misspecified Linear Models With Multicollinearity. Ann. Inst. Statist. Math. 1989;41(4):717--724. %doi:10.1007/BF00057737
\bibitem{Sir15}%15
$\c{S}$iray~GÜ. r-d Class Estimator under misspecification. Communications in Statistics-Theory and Methods. 2015;44(22):4742--4756. %doi:10.1080/03610926.2013.835421
\bibitem{Wu16}%16
Wu~J. Superiority of the r-k class estimator over some estimators in a misspecified linear model. Communication in Statistics-Theory and Methods. 2016;45:1453--1458.
%doi:10.1080/03610926.2013.863934
\bibitem{Ch17}%17
Chandra~S, Tyagi~G. On the performance of some biased estimators in a misspecified model with correlated regressors. STATISTICS IN TRANSITION new series. 2017;27--52.
%doi:10.21307/stattrans-2016-056
\bibitem{Kay17a}%18
Kayanan~M, Wijekoon~P. Performance of Existing Biased Estimators and the respective Predictors in a Misspecified Linear Regression Model. Open Journal of Statistics. 2017;876--900.
%doi:10.4236/ojs.2017.75062
\bibitem{Te80}%19
Ter\"{a}svirta~T. Linear restrtctions in misspecified linear models and polynomial distributed lag estimation. Finland : Departntent of Statistics University of Helsinki; 1980.
\bibitem{Mt81}%20
Mittelhammer~RC. On specification error in the general linear model and weak mean square error superiority of the mixed estimator. Communications in Statistics-Theory and Methods. 1981;167--176.
%doi:10.1080/03610928108828027
\bibitem{Oh84}%21
Ohtani~K, Honda~Y. On small sample properties of the mixed regression predictor under misspecification. Communications in Statistics-Theory and Methods. 1984;2817--2825.
%doi:10.1080/03610928408828863
\bibitem{Ka86}%22
Kadiyala~K. Mixed Regression Estimator under misspecification. Economic Letters. 1986;21:27--30.
%doi:10.1016/0165-1765(86)90115-1
\bibitem{Tr89}%23
Trenkler~G, Wijekoon~P. Mean square error matrix superiority of the mixed regression estimator under misspecification. Statistica.anno. 1989;49(1):65--71. %doi:10.6092/issn.1973-2201/785
\bibitem{Wt89}%24
Wijekoon~P, Trenkler~G. Mean Square Error Matrix Superiority of Estimators under Linear Restrictions and Misspecification. Economics Letters. 1989;30:141--149.
%doi:10.1016/0165-1765(89)90052-9
\bibitem{Kay17b}%25
Kayanan~M, Wijekoon~P. Stochastic Restricted Biased Estimators in misspecified regression model with incomplete prior information. Journal of Probability and Statistics. 2018.
\bibitem{Ar15}%26
Arumairajan~S, Wijekoon~P. Optimal Generalized Biased Estimator in Linear Regression Model. Open Journal of Statistics. 2015;5:403--411.
%doi:10.4236/ojs.2015.55042
\bibitem{ro95}%27
Rao~CR, Toutenburg~H. Linear Models :Least Squares and Alternatives. 2nd ed. New York (NY): Springer-Verlag; 1995.
\bibitem{new71}%28
Newhouse~JP, Oman~SD. An Evaluation of Ridge Estimators. Santa Monica(CA): RAND Corporation; 1971.
\bibitem{Mc75}%29
McDonald~GC, Galarneau~DI. A Monte Carlo Evaluation of Some Ridge-Type Estimators. Journal of the American Statistical Association. 1975;70:407--416.
\bibitem{Gr98}%30
Gruber~M. Improving Efficiency by Shrinkage: The James-Stein and Ridge Regression Estimators. New York (NY): CRC Press; 1998.
\bibitem{Ak03}%31
Akdeniz~F, Erol~H. Mean Squared Error Matrix Comparisons of Some Biased Estimators in Linear Regression. Communications in Statistics-Theory and Methods. 2003;2389--2413.
%doi:10.1081/STA-120025385
\bibitem{Tr90}%32
Trenkler~G, Toutenburg~H. Mean square error matrix comparisons between biased estimators an overview of recent results. Statistical Papers. 1990;31:165-179.
%doi:10.1007/BF02924687
\end{thebibliography}
\end{document}